\documentclass[12pt]{amsart}

\usepackage{amsthm,amsmath,amsfonts,amssymb}
\usepackage{mathtools}
\usepackage{mathcomp}
\usepackage{nicefrac}
\usepackage{pxfonts}
\usepackage[all]{xy}
\usepackage{graphicx}
\usepackage{array}
\usepackage[DIV10, headinclude, BCOR 10mm]{typearea}
\usepackage{latexsym}
\usepackage{paralist} 

\usepackage{hyperref}
\hypersetup{pdfauthor=Simon Beier,
pdftitle=An integral lift starting in odd Khovanov homology of Szab\'o's spectral sequence}

\DeclareGraphicsExtensions{.eps,.pdf,.jpg}
\title[A spectral sequence in odd Khovanov homology]{An integral lift, starting in odd Khovanov homology, of Szab\'o's spectral sequence}
\author{Simon Beier}
\address{Fakult\"at f\"ur Mathematik, Universit\"at Regensburg,
  93040 Regensburg, Germany}
\email{simon.beier@mathematik.uni-regensburg.de}

\newcommand{\eps}{\varepsilon}

\newcommand{\V}{\textup{V}}
\newcommand{\defeq}{\mathrel{\vcentcolon=}}

\newcommand{\z}{\mathbb{Z}}
\newcommand{\n}{\mathbb{N}}
\newcommand{\s}{\mathbb{S}}
\newcommand{\C}{\mathcal{C}}
\newcommand{\D}{\mathcal{D}}
\newcommand{\te}{\textup{E}}
\newcommand{\hh}{\widehat{\textup{H}}}
\newcommand{\td}{\textup{d}}
\newcommand{\f}{\mathfrak{F}}
\newcommand{\tr}{\textup{R}}
\newcommand{\tsp}{\textup{sp}}
\newcommand{\ta}{\textup{A}}
\newcommand{\tb}{\textup{B}}
\newcommand{\tc}{\textup{C}}
\newcommand{\tde}{\textup{D}}
\newcommand{\tf}{\textup{F}}
\newcommand{\tg}{\textup{G}}
\newcommand{\ttt}{\textup{T}}
\newcommand{\tit}{\tau\in\textup{T}}
\newcommand{\gr}{\textup{gr}}
\newcommand{\hc}{\widehat{\textup{C}}}
\newcommand{\tth}{\textup{H}}

\newtheorem{thm}{Theorem}[section]
\newtheorem{prop}[thm]{Proposition}
\newtheorem{lem}[thm]{Lemma}
\newtheorem{cor}[thm]{Corollary}

\theoremstyle{remark}

\theoremstyle{definition}
\newtheorem{defn}[thm]{Definition}
\newtheorem{quest}[thm]{Question}
\newtheorem*{ack}{Acknowledgment}

\begin{document}

\begin{abstract}
Ozsv\'ath, Rasmussen and Szab\'o constructed odd Khovanov homology \cite{ors}. It is a link invariant which has the same reduction modulo~2 as (even) Khovanov homology. Szab\'o introduced a spectral sequence with mod 2 coefficients from mod 2 Khovanov homology to another link homology \cite{szabo}. He got his spectral sequence from a chain complex with a filtration. We give an integral lift of Szab\'o's complex that provides a spectral sequence from odd Khovanov homology to a link homology, from which one can get Szab\'o's link homology with the Universal Coefficient Theorem. Szab\'o has constructed such a lift independently (unpublished).
\end{abstract}

\maketitle

\section{Introduction}
Khovanov homology $\textup{Kh}(\cdot)$ is a link invariant which arises as the homology of a bigraded chain complex with $\z$-coefficients assigned to a link diagram \cite{kh,barnat}. The graded Euler characteristic of Khovanov homology is the Jones polynomial \cite{barnat}.

Ozsv\'ath, Rasmussen and Szab\'o \cite{ors} constructed a modified version of Khovanov homology, odd Khovanov homology $\textup{Kh}'(\cdot)$. Let us briefly recall their construction. To assign a chain complex to a link diagram $\D$ we look at the hypercube of resolutions of $\D$, like for even Khovanov homology. At each vertex in the hypercube we get a collection of embedded, planar circles and we assign to the vertex the exterior algebra of the free abelian group generated by the circles. To define a differential two kinds of extra data are used.
\begin{enumerate}
\item[1.]
At each crossing of $\D$ an orientation of the arc that connects the two segments of the 0-resolution, see Figure \ref{arc}.
\begin{figure}[htbp]
\centering
\includegraphics{gra.1}
\caption{}
\label{arc}
\end{figure}
\item[2.]
An edge assignment $\eps$. That is a map that assigns to each edge in the hypercube an element of $\{-1,1\}$ and satisfies some additional properties.
\end{enumerate}
According to which properties $\eps$ satisfies, we call $\eps$ of type~X or of type~Y. For each edge $a$ in the hypercube we get a homomorphism between the exterior algebras at the boundary points of $a$: At $a$ we get one of the two pictures in Figure \ref{1dim} and each of these pictures gives us a homomorphism which we multiply with $\eps(a)$. The homomorphisms at the edges give us the differential for the odd Khovanov complex and the properties that $\eps$ has to  fulfill guarantee that we have $d\circ d=0$.

If we choose for a given link diagram two different orientations of the arcs of the 0-resolution and two different edge assignments of the same type (X or Y), then the two resulting chain complexes are isomorphic. If the two edge assignments are of opposite types, then Ozsv\'ath, Rasmussen and Szab\'o claim that the two resulting chain complexes are also isomorphic \cite[Lemma 2.4]{ors}. But it seems that they do not give an isomorphism. From the Reidemeister invariance of odd Khovanov homology we conclude the following:
\begin{prop}\label{xycor}
Let $L$ be a link. Then there is an isomorphism $\textup{Kh}_X'(L)\cong\textup{Kh}_Y'(L)$ between the odd Khovanov homology groups of type X and type Y.
\end{prop}
The mod 2 reductions of even and odd Khovanov homology are the same and so the graded Euler characteristic of odd Khovanov homology is the Jones polynomial as in the even case.

We now prepare our main result (Theorem \ref{mainth}): There exists an integral lift of a mod 2 chain complex constructed by Szab\'o \cite{szabo} which provides a spectral sequence starting in odd Khovanov homology. So let us summarize Szab\'o's result: He constructed a spectral sequence $\te(\cdot;\z_2)$ with mod 2 coefficients which is a link invariant. In this article all spectral sequences start with the $\te^2$-term. The $\te^2$-term of Szab\'o's spectral sequence is isomorphic to mod 2 Khovanov homology $\te^2(\cdot;\z_2)\cong\textup{Kh}(\cdot;\z_2)$. The spectral sequence $\te(\cdot;\z_2)$ converges to a link homology $\hh(\cdot;\z_2)$. To get Szab\'o's spectral sequence we define a differential $\td=\sum_{n\in\n}\td_n$ on the mod 2 Khovanov chain groups. The map~$\td_n$ shifts the $(\textup{h},\delta)$-grading by $(n,-2)$ and $\td_1$ equals the mod 2 Khovanov differential. From the h-grading we get a filtration on the complex and so we get the described spectral sequence, which converges to the homology of the complex. Like for odd Khovanov homology we use an orientation of the arcs of the 0-resolution to define the differential. The isomorphism type of the complex is independent of this choice of orientation.

We construct $\td_n$ as follows: At each $n$-dimensional face $a$ in the hypercube of resolutions we get a collection of embedded, planar circles connected by $n$ oriented arcs. The type of this picture determines a homomorphism (which could be trivial) between the groups at the two vertices of $a$ where either all crossings belonging to $a$ are 0-smoothed or all of them are 1-smoothed. Then $\td_n$ is the sum over all these homomorphisms at $n$-dimensional faces.

Theorem \ref{mainth} is our main result. It was also proved by Szab\'o independently and earlier (unpublished).
\begin{thm}\label{mainth}
There exists an integral lift $\hc$ of Szab\'o's complex, so that $\hc$ with only the $\td_1$ differential is isomorphic to the odd Khovanov complex of type Y. The resulting spectral sequence and the homology groups of $\hc$ are link invariants.
\end{thm}
\begin{cor}\label{maincor}
For a link $L$ we get a spectral sequence $\te(L)$. The $\te^2$-term is isomorphic to odd Khovanov homology $\te^2(L)\cong\textup{Kh}'(L)$. The spectral sequence $\te(L)$ converges to a link homology $\hh(L)$, from which one can get Szab\'o's $\hh(L;\z_2)$ with the Universal Coefficient Theorem.
\end{cor}

So we have constructed link invariants $\te(\cdot)$ and $\hh(\cdot)$ such that $\te(\cdot)$ contains potentially more information than odd Khovanov homology $\textup{Kh}'(\cdot)$ on the one hand and $\hh(\cdot)$ contains potentially more information than Szab\'o's link homology $\hh(\cdot;\z_2)$ on the other hand. It would be worthwhile to have a computer program which computes $\te(\cdot)$ and $\hh(\cdot)$ to answer the following natural questions:
\begin{quest}
Which links $L_1,L_2$ with $\textup{Kh}'(L_1)\cong\textup{Kh}'(L_2)$ can be distinguished by $\te(\cdot)$ or $\hh(\cdot)$?
\end{quest}
\begin{quest}
Which links $L_1,L_2$ with $\hh(L_1;\z_2)\cong\hh(L_2;\z_2)$ or even $\te(L_1;\z_2)\cong\te(L_2;\z_2)$ can be distinguished by $\hh(\cdot)$ or $\te(\cdot)$?
\end{quest}

Another interesting aspect of our result is that we have constructed a combinatorial defined link homology $\hh(\cdot)$ that is related to odd Khovanov homology $\textup{Kh}'(\cdot)$ via the spectral sequence $\te(\cdot)$. If we have proved that $\hh(\cdot)$ or $\textup{Kh}'(\cdot)$ has a particular property it could be possible to prove by means of $\te(\cdot)$ that the other link homology $\textup{Kh}'(\cdot)$ or $\hh(\cdot)$ respectively also has this property. For example all links that can be distinguished from the unknot by $\hh(\cdot)$ can also be distinguished from the unknot by $\textup{Kh}'(\cdot)$, which can easily be proved via $\te(\cdot)$.

Szab\'o gives a second version of his complex with differential $\td '$ \cite[beginning of Section 8]{szabo}.
\begin{prop}\label{primelem}
Theorem \ref{mainth} holds also if we replace "Szab\'o's complex" by "Szab\'o's complex with differential $\td '$" and "type Y" by "type X".
\end{prop}
Let $\te '(\cdot)$ and $\hh '(\cdot)$ be the spectral sequence and the link homology we get from Proposition \ref{primelem}. From the Reidemeister invariance of the spectral sequence and the homology we conclude analogously to Proposition \ref{xycor}:
\begin{cor}\label{primecor}
Let $L$ be a link. Then there are isomorphisms $\te '(L)\cong\te(L)$ and $\hh '(L)\cong\hh(L)$.
\end{cor}
Let $\te '(\cdot;\z_2)$ be the spectral sequence induced by Szab\'o's complex with differential $\td '$. With Corollary \ref{primecor} we see the following, which was conjectured by Seed \cite[Conjecture 4.14]{seed}.
\begin{cor}\label{primecor2}
Let $L$ be a link. Then there is an isomorphism $\te '(L;\z_2)\cong\te(L;\z_2)$.
\end{cor}

This article is organized as follows. In Section \ref{sodd} we briefly review the construction of odd Khovanov homology. We also prove Proposition \ref{xycor} (Corollary \ref{xycor2}). In Section \ref{sspec} we construct the complex of Theorem \ref{mainth} and give the stated properties. We also prove Proposition \ref{primelem} and Corollary \ref{primecor}. In Section \ref{sdep} we analyze how the differential of our complex depends on the orientation of the arcs. This leads to the proof that the isomorphism class of our complex is independent of the orientation of the arcs. The result of Section \ref{sdep} also helps us to prove that our differential really satisfies $d\circ d=0$. The author has done this proof in his master thesis \cite{master}. It is sketched in Section \ref{sdd}.
\begin{ack}
This is a reworked version of my master thesis which I wrote under supervision of Thomas Schick at Georg August University G\"ottingen in summer 2011. I wish to thank Thomas for suggesting this topic and for the freedom I enjoyed during the project.

I would like to thank Clara L\"oh for giving me many hints how to make this article more readable.
\end{ack}

\section{Odd Khovanov homology}\label{sodd}
We briefly review the construction of odd Khovanov homology \cite{ors}. For a link in $\s^3$ we get a diagram $\D$ in $\s^2$ with $n$ crossings \cite{buzie}. We get the 0-resolution of $\D$ as follows. For each crossing $x$ we change $\D$ in a small neighborhood of $x$ as in Figure \ref{0res}.
\begin{figure}[htbp]
\centering
\includegraphics{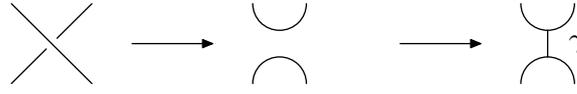}
\caption{0-resolution}
\label{0res}
\end{figure}
We call $\gamma$ the arc connecting the two segments of the 0-resolution of $x$. The 0-resolution of $\D$ is a collection of disjoint embedded circles in $\s^2$ connected by $n$ arcs. This leads us to the following definition.
\begin{defn}
An \emph{$n$-dimensional (oriented) configuration} is an equivalence class of a set of disjoint circles in $\s^2$ and $n$ disjoint embedded (oriented) arcs such that the boundary points of the arcs lie on the circles and the interiors of the arcs are disjoint from the circles. Two such sets are equivalent if there is an orientation preserving diffeomorphism of $\s^2$ that maps one set to the other one. We assume that the arcs of an $n$-dimensional configuration are numbered from 1 to $n$.
\end{defn}
With this definition the 0-resolution of $\D$ (with numbered crossings) is an $n$-dimensional configuration.
\begin{defn}
Let $\C$ be an oriented configuration.
\begin{enumerate}
\item[-] 
The unoriented configuration $\overline{\C}$ is the configuration one gets by forgetting the orientation on the arcs.
\item[-] 
The dual configuration $\C^*$ is the oriented configuration one gets by changing $\C$ in small neighborhoods of the arcs as in Figure \ref{cstar}. So the arcs are rotated by 90 degrees counter-clockwise.
\begin{figure}[htbp]
\centering
\includegraphics{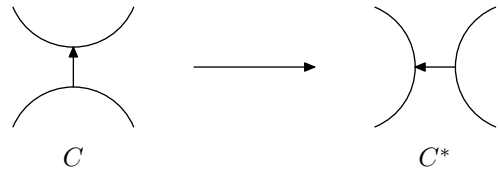}
\caption{$C\rightarrow C^*$}
\label{cstar}
\end{figure}
\item[-] 
The reverse configuration $\textup{r}(\C)$ is the oriented configuration one gets by reversing the orientation of all arcs. We have $(\C^*)^*=\textup{r}(\C)$.
\item[-] 
The mirror configuration $\textup{m}(\C)$ is the oriented configuration one gets by reversing the orientation of $\s^2$. We have $(\textup{m}(\C^*))^*=\textup{m}(\C)$.
\end{enumerate}
\end{defn}
\begin{defn}
Let $\C$ be an oriented configuration.
\begin{enumerate}
\item[-]
The circles of $\C$ are also called \emph{starting circles} of $\C$.
\item[-]
The circles of $\C^*$ are called \emph{ending circles} of $\C$.
\item[-]
The circles of $\C$ that are disjoint from all arcs are called \emph{passive circles} of $\C$. The passive circles of $\C$ build the 0-dimensional configuration~$\textup{pass}(\C)$.
\item[-]
The other circles of $\C$ are called \emph{active circles} of $\C$. We get the active part $\textup{act}(\C)$ of $\C$ by omitting all the passive circles.
\end{enumerate}
We call $\C$ \emph{active} if it is equal to $\textup{act}(\C)$, which means that $\C$ has no passive circles.
\end{defn}
The passive circles of $\C$ are equal to the passive circles of $C^*$. To define the odd Khovanov complex we need to look at the hypercube of resolutions. So we make the following definition.
\begin{defn}
For $k\le n\in\n_0$ the set of $k$-dimensional faces of the hypercube $[0,1]^n$ is
\[\f(n,k)\defeq\left\{a\in\{0,1,*\}^n\ \ \Bigm|\ \ \#\{i\in\{1,\dots,n\}\ \ |\ \ a_i=*\}=k\right\}.\]
So $\f(n,0)$ and $\f(n,1)$ are the vertices and edges of $[0,1]^n$. For $a\in\f(n,k)$ we get $a^0,\ a^1\in\f(n,0)$ from $a$ by replacing all * with 0 or 1 respectively. We call $a^0$ and $a^1$ the \emph{starting} and \emph{end point} of $a$.
\end{defn}
\begin{defn}
Let $\C$ be an $n$-dimensional oriented configuration. For all $a\in\f(n,k)$ we define the resolution $\tr(\C,a)$ to be the $k$-dimensional oriented configuration that we get from $\C$ as follows: For all $i\in\{1,\dots,n\}$ with $a_i=1$ we change $\C$ in a small neighborhood of the $i$-th arc as in Figure \ref{cstar}. Then we remove all arcs for which we have $a_i\ne*$.
\end{defn}
With this notation we can now define the odd Khovanov chain groups.
\begin{defn}
Let $\C$ be an $n$-dimensional oriented configuration. Then $\textup{V}(\C)$ is the free abelian group generated by the circles of $\C$. The \emph{odd Khovanov chain group} of $\C$ is the abelian group
\[\textup{C}(\C)\defeq\bigoplus_{a\in\f(n,0)}\Lambda\textup{V}(\tr(\C,a)),\]
where $\Lambda$ stands for the exterior algebra.
\end{defn}
Next we want to define the odd Khovanov differential that consists of maps at the edges $b\in\f(n,1)$ which depend on $\tr(\C,b)$. So we need to define $\z$-module~homomorphisms
\[\partial_\C:\Lambda\textup{V}(\C)\to\Lambda\textup{V}(\C^*)\]
for all active oriented one-dimensional configurations $\C$. For $\C$ there are the two possibilities in Figure \ref{1dim}.
\begin{figure}[htbp]
\centering
\includegraphics{gra.2}
\caption{}
\label{1dim}
\end{figure}
Then $\C^*$ is the other possibility respectively.
\begin{defn}
If $\C$ is a join we define
\[\partial_\C(1)=1,\qquad\partial_\C(x_1)=y,\qquad
\partial_\C(x_2)=y,\qquad\partial_\C(x_1\wedge x_2)=0.\]
If $\C$ is a split we define
\[\partial_\C(1)=x_1-x_2,\qquad\partial_\C(y)=x_1\wedge x_2.\]
\end{defn}
Next we want to define $\partial_\C$ for the case when $\C$ is not necessarily active.
\begin{defn}\label{nact}
For $m\in\n_0$ and passive circles $z_1,\dots,z_m$ of $\C$ set $\omega\defeq z_1\wedge\cdots\wedge z_m$. If $\textup{act}(\C)$ is a join we define
\[\partial_\C(\omega)=\omega,\qquad\partial_\C(x_1\wedge\omega)=
y\wedge\omega,\qquad\partial_\C(x_2\wedge\omega)=y\wedge\omega,
\qquad\partial_\C(x_1\wedge x_2\wedge\omega)=0.\]
If $\textup{act}(\C)$ is a split we define
\[\partial_\C(\omega)=(x_1-x_2)\wedge\omega,\qquad
\partial_\C(y\wedge\omega)=x_1\wedge x_2\wedge\omega.\]
\end{defn}
Now let $\C$ be an oriented $n$-dimensional configuration and $a\in\f(n,1)$ an edge in the hypercube of resolutions. Then we get a $\z$-module~homomorphism
\[\partial_{\tr(\C,a)}:\Lambda\underbrace{\V(\tr(\C,a))}_{=\V(\tr(\C,a^0))}\to
\Lambda\underbrace{\V((\tr(\C,a))^*)}_{=\V(\tr(\C,a^1))}.\]
From these maps we want to get a differential on $\textup{C}(\C)$. Therefore we have to analyze, how the four edges of a square in the hypercube of resolutions behave. Let $\C$ be an oriented 2-dimensional configuration and
\[a\defeq(*,0),\qquad b\defeq(1,*),\qquad c\defeq(0,*),\qquad d\defeq(*,1).\]
What is the relation between $\partial_{\tr(\C,b)}\circ\partial_{\tr(\C,a)}$ and $\partial_{\tr(\C,d)}\circ\partial_{\tr(\C,c)}$? If
\[0\ne\partial_{\tr(\C,b)}\circ\partial_{\tr(\C,a)}
=\partial_{\tr(\C,d)}\circ\partial_{\tr(\C,c)}\]
we call $\C$ of \emph{type K}. If
\[0\ne\partial_{\tr(\C,b)}\circ\partial_{\tr(\C,a)}
=-\partial_{\tr(\C,d)}\circ\partial_{\tr(\C,c)}\]
we call $\C$ of \emph{type A}. One can easily check the following: If $\textup{act}(\C)$ has three or four circles, then $\C$ is of type K. If $\textup{act}(\C)$ has two circles and is disconnected, then $\C$ is of type A. The other possibilities for $\textup{act}(\C)$ for type K and A are given in Figure \ref{f1} and Figure \ref{f2}.
\begin{figure}[htbp]
\centering
\includegraphics{gra.13}
\caption{type K}
\label{f1}
\end{figure}
\begin{figure}[htbp]
\centering
\includegraphics{gra.14}
\caption{type A}
\label{f2}
\end{figure}
If the arcs in this figures are not oriented, this means that the orientation can be chosen arbitrarily. There are two possibilities for $\textup{act}(\C)$ where $\C$ is neither of type K nor of type A. They are given in Figure \ref{fxy} and are called of \emph{type X} and \emph{Y} respectively.
\begin{figure}[htbp]
\centering
\includegraphics{gra.15}
\caption{}
\label{fxy}
\end{figure}
In this cases we have
\[0=\partial_{\tr(\C,b)}\circ\partial_{\tr(\C,a)}
=\partial_{\tr(\C,d)}\circ\partial_{\tr(\C,c)}.\]
To define the odd Khovanov differential we need the following extra data.
\begin{defn}
Let $n\in\n_0$. A map $\eps:\f(n,1)\to\{-1,1\}$ is called an \emph{edge assignment}. For a square $s\in\f(n,2)$ let $a,\ b,\ c,\ d\in\f(n,1)$ be the edges of $s$. Define $p(s,\eps)\defeq\eps(a)\eps(b)\eps(c)\eps(d)$.
\end{defn}
\begin{defn}\label{defxy}
Let $\C$ be an $n$-dimensional oriented configuration and $\eps$ an edge assignment. We call $\eps$ of \emph{type X} with respect to $\C$ if for all $s\in\f(n,2)$ we have
\[p(s,\eps)=\begin{cases}
1\qquad&\textup{if $\tr(\C,s)$ is of type A or X,}\\
-1\qquad&\textup{if $\tr(\C,s)$ is of type K or Y.}
\end{cases}\]
We call $\eps$ of \emph{type Y} with respect to $\C$ if for all $s\in\f(n,2)$ we have
\[p(s,\eps)=\begin{cases}
1\qquad&\textup{if $\tr(\C,s)$ is of type A or Y,}\\
-1\qquad&\textup{if $\tr(\C,s)$ is of type K or X.}
\end{cases}\]
\end{defn}
\begin{thm}
For every oriented configuration $\C$ there is an edge assignment of type X and an edge assignment of type Y.
\end{thm}
\begin{proof}
A proof is given by Ozsv\'ath, Rasmussen and Szab\'o \cite[Lemma 1.2]{ors}.
\end{proof}
\begin{defn}
Let $\C$ be an oriented configuration and $\eps$ an edge assignment of type X or Y with respect to $\C$. Then the \emph{odd Khovanov complex} is
\[\textup{C}(\C,\eps)\defeq\biggl(\textup{C}(\C)=\bigoplus_{a\in\f(n,0)}
\Lambda\V(\tr(\C,a)),\ \partial(\C,\eps)\defeq\bigoplus_{b\in\f(n,1)}
\eps(b)\partial_{\tr(\C,b)}\biggr).\]
\end{defn}
With Definition \ref{defxy} we get:
\begin{prop}
The composition $\partial(\C,\eps)\circ\partial(\C,\eps)$ is trivial and so $\textup{C}(\C,\eps)$ is indeed a chain complex.
\end{prop}
Next we want to define a bigrading on the odd Khovanov complex.
\begin{defn}
Let $\C$ be an oriented $n$-dimensional configuration. Then $|\C|$ is the number of circles of $\C$. For $a\in\f(n,0)$ we define $|a|\defeq\sum_{i=1}^na_i$. We get the \emph{h-grading} on $\textup{C}(\C)$ by
\[\textup{C}(\C)^\textup{h}\defeq\bigoplus_{a\in\f(n,0),|a|=\textup{h}}\Lambda\V(\tr(\C,a))\]
for $\textup{h}\in\z$ and the \emph{$\delta$-grading} by
\[\textup{C}(\C)_\delta\defeq\bigoplus_{m\in\mathbb{N}_0,a\in\f(n,0):\ |\tr(\C,a)|-2m-|a|=\delta}\Lambda^m\V(\tr(\C,a))\]
for $\delta\in\z$.
Furthermore let $\textup{C}(\C)_{\textup{h},\delta}\defeq\textup{C}(\C)^\textup{h}\cap\textup{C}(\C)_\delta$.
\end{defn}
We have $\textup{C}(\C)=\bigoplus_{\textup{h},\delta\in\z}\textup{C}(\C)_{\textup{h},\delta}$. The differential $\partial(\C,\eps)$ raises the h-grading by 1 and decreases the $\delta$-grading by 2. The next theorem shows that $\textup{C}(\C,\eps)$ depends only on $\overline{\C}$ and the type of $\eps$.
\begin{thm}
Let $\C$ and $\D$ be oriented configurations with $\overline{\C}=\overline{\D}$. Let $\eps$ and $\etaup$ be edge assignments, so that $\eps$ has the same type (X or Y) with respect to $\C$ as $\etaup$ with respect to $\D$. Then $\textup{C}(\C,\eps)$ and $\textup{C}(\D,\etaup)$ are isomorphic as bigraded chain complexes.
\end{thm}
\begin{proof}
A proof is given by Ozsv\'ath, Rasmussen and Szab\'o \cite[Lemma 2.2 and Lemma 2.3]{ors}.
\end{proof}
Ozsv\'ath, Rasmussen and Szab\'o claim that $\textup{C}(\C,\eps)$ and $\textup{C}(\D,\etaup)$ are also isomorphic if $\eps$ with respect to $\C$ and $\etaup$ with respect to $\D$ are of opposite types \cite[Lemma 2.4]{ors}. But it seems that they do not give an isomorphism. In Corollary~\ref{xycor2} we will conclude from the Reidemeister invariance of odd Khovanov homology that the homology groups of $\textup{C}(\C,\eps)$ and $\textup{C}(\D,\etaup)$ are isomorphic if $\eps$ with respect to $\C$ and $\etaup$ with respect to $\D$ are of opposite types.

Next we want to define the odd Khovanov complex for link diagrams.
\begin{defn}
Let $\D$ be a link diagram. Then $\textup{n}_+(\D)$ and $\textup{n}_-(\D)$ are the number of plus and minus crossings in $\D$, see Figure \ref{f+-}.
\end{defn}
\begin{figure}[htbp]
\centering
\includegraphics{gra.5}
\caption{}
\label{f+-}
\end{figure}
\begin{defn}\label{defkd}
Let $\C$ be an oriented configuration, so that $\overline{\C}$ is a 0-resolution of $\D$. Then we define $\textup{C}(\D,\C)^\textup{h}\defeq\textup{C}(\C)^{\textup{h}+\textup{n}_-(\D)}$ and $\textup{C}(\D,\C)_\delta\defeq\textup{C}(\C)_{\delta-\textup{n}_+(\D)}$. Furthermore let $\eps$ be an edge assignment of type X or Y with respect to $\C$. Then the \emph{odd Khovanov complex} $\textup{C}(\D,\C,\eps)$ is the bigraded chain complex that only differs from $\textup{C}(\C,\eps)$ in the bigrading as described.
\end{defn}
The odd Khovanov homology is Reidemeister invariant in the following sense.
\begin{thm}
Let $\D_1$ and $\D_2$ be link diagrams that differ by finitely many Reidemeister moves. Then there exist oriented configurations $\C_1$, $\C_2$, so that $\overline{\C}_1$, $\overline{\C}_2$ are 0-resolutions of $\D_1$, $\D_2$ and edge assignments $\eps_1$, $\eps_2$, so that $\eps_1$ with respect to $\C_1$ and $\eps_2$ with respect to $\C_2$ have type X and so that the homology groups $\textup{H}_{\textup{h},\delta}(\textup{C}(\D_1,\C_1,\eps_1))$ and $\textup{H}_{\textup{h},\delta}(\textup{C}(\D_2,\C_2,\eps_2))$ are isomorphic for all $\textup{h},\delta\in\z$. The analogous statement holds for type Y.
\end{thm}
\begin{proof}
A proof is given by Ozsv\'ath, Rasmussen and Szab\'o \cite[Proposition 3.1, Proposition 3.2, and Proposition 3.3]{ors}.
\end{proof}
So we assign to a link $L$ two isomorphism classes of bigraded abelian groups $\textup{Kh}_\textup{X}'(L)$ and $\textup{Kh}_\textup{Y}'(L)$, odd Khovanov homology of type X and of type Y. The next result shows that $\textup{Kh}_\textup{X}'(L)$ and $\textup{Kh}_\textup{Y}'(L)$ are the same.
\begin{cor}\label{xycor2}
Let $\C$ be an oriented configuration. Let $\eps$ be an edge assignment of type X and $\etaup$ be an edge assignment of type Y with respect to $\C$. Then the homology groups $\textup{H}_{\textup{h},\delta}(\textup{C}(\C,\eps))$ and $\textup{H}_{\textup{h},\delta}(\textup{C}(\C,\etaup))$ are isomorphic for all $\textup{h},\delta\in\z$.
\end{cor}
\begin{proof}
Let $n$ be the dimension of $\C$. Then for all squares $s\in\f(n,2)$ we have: If $\tr(\C,s)$ is of type A, K, X, Y then $\tr(\textup{r}(\textup{m}(\C)),s)$ is of type A, K, Y, X. So $\eps$ is of type Y with respect to $\textup{r}(\textup{m}(\C))$. One easily sees: The canonical isomorphism of the bigraded abelian groups $\textup{C}(\C)$ and $\textup{C}(\textup{r}(\textup{m}(\C)))$ is actually an isomorphism between the chain complexes $\textup{C}(\C,\eps)$ and $\textup{C}(\textup{r}(\textup{m}(\C)),\eps)$.

Let $\D_1$ be a link diagram so that $\overline{\C}$ is the 0-resolution of $\D_1$. Let $\D_2$ be the link diagram that we get from $\D_1$ if we reverse the orientation of $\s^2$ and change top and down for all crossings. Let $L$ be a link that is mapped to $\D_1$ by orthogonal projection. If we reverse the direction of the projection $L$ is mapped to $\D_2$. So $\D_1$ and $\D_2$ are diagrams of the same link and hence $\D_2$ differs from $\D_1$ by finitely many Reidemeister moves. Furthermore $\overline{\textup{m}(\C)}$ is the 0-resolution of $\D_2$. It follows that
\begin{multline*}
\textup{H}_{\textup{h},\delta}(\textup{C}(\C,\eps))\cong \textup{H}_{\textup{h},\delta}(\textup{C}(\textup{r}(\textup{m}(\C)),\eps))\cong \textup{H}_{\textup{h}-\textup{n}_-(\D_2),\delta+\textup{n}_+(\D_2)}(\textup{C}(\D_2,\textup{r}(\textup{m}(\C)),\eps))\\
\cong \textup{H}_{\textup{h}-\textup{n}_-(\D_1),\delta+\textup{n}_+(\D_1)}(\textup{C}(\D_1,\C,\etaup))\cong \textup{H}_{\textup{h},\delta}(\textup{C}(\C,\etaup)).\qedhere
\end{multline*}
\end{proof}
So we have proved Proposition \ref{xycor}. The reductions modulo 2 of the even and the odd Khovanov complex are the same \cite[Proof of Proposition 1.6]{ors}:
\begin{thm}\label{mod2th}
Let $\D$ be a link diagram, $\C$ an oriented configuration so that $\overline{\C}$ is the 0-resolution of $\D$, and $\eps$ an edge assignment of type X or Y with respect to $\C$. Then the bigraded chain complex $\textup{C}(\D,\C,\eps)\otimes_\z\z_2$ is isomorphic to the even Khovanov chain complex of $\D$ with $\z_2$ coefficients with $(\textup{h},\delta)$-grading.
\end{thm}
With $\textup{q}=\delta+2\textup{h}$ we get another grading for the even and for the odd Khovanov chain complex. The differential does not change the q-grading. If we build the Euler characteristic with respect to the h-grading of the even Khovanov homology of a link we get a Laurent polynomial through the q-grading. This is the unnormalized Jones polynomial \cite[Theorem 1]{barnat}. From Theorem~\ref{mod2th} it follows that this is also true for odd Khovanov homology.

\section{Construction of the spectral sequence}\label{sspec}
In this Section we construct our spectral sequence $\te(L)$ for a link $L$ (see Corollary \ref{maincor}). To do this we give an integral lift of Szab\'o's complex \cite{szabo}. First we need to fix some combinatorial notation.
\begin{defn}
For $k\le n\in\n_0$ a \emph{path of edges with length $k$} is a $k$-tuple $\theta=({}_1\theta,\dots,{}_k\theta)\in\left(\f(n,1)\right)^k$ with ${}_i\theta^1={}_{i+1}\theta^0$ for all $i\in\{1,\dots,k-1\}$. We say \emph{$\theta$ goes from ${}_1\theta^0$ to ${}_k\theta^1$}. For an edge assignment $\eps:\f(n,1)\to\{-1,1\}$ and a path of edges $\theta\in\left(\f(n,1)\right)^k$ we define $\eps(\theta)\defeq\prod_{i=1}^k\eps({}_i\theta)$.
\end{defn}
\begin{defn}\label{spdef}
For an $n$-dimensional configuration $\C$ and $a\in\f(n,0)$ let the \emph{split of $\C,a$} be
\[\tsp(\C,a)\defeq\frac{|\tr(\C,a)|-|C|+|a|}{2}.\]
For a path of edges $\theta\in\left(\f(n,1)\right)^k$ let the \emph{split of $\C,\theta$} be
\[\tsp(\C,\theta)\defeq\prod_{i=1}^{k-1}(-1)^{\tsp(\C,{}_i\theta^1)}.\]
\end{defn}
\begin{lem}
In the situation of Definition \ref{spdef} let $\theta$ be a path of edges from $(0,\dots,0)$ to $a$. Then we have $\textup{sp}(\C,a)=\#\left\{i\in\{1,\dots,|a|\}\ \ |\ \ \tr(\C,{}_i\theta)\textup{ is split}\right\}$.
\end{lem}
\begin{proof}
This is proved by induction on $|a|$.
\end{proof}
\begin{defn}
For $a\in\f(n,1)$ let $\overline{a}\in\{1,\dots,n\}$ be the position of * in $a$.
\end{defn}
\begin{defn}
For $n\in\n_0$ we define the bijection
\begin{align*}
\phi:\left\{\theta\in\left(\f(n,1)\right)^n\ \ |\ \ \theta\textup{ is a path of edges}\right\}&\to\textup{S}_n,\\
\theta&\mapsto(i\mapsto\overline{{}_i\theta}),
\end{align*}
where $\textup{S}_n$ stands for the symmetric group. We have $\phi(\theta)\in\textup{S}_n$ because $\theta$ goes from $(0,\dots,0)$ to $(1,\dots,1)$.
\end{defn}
\begin{defn}
For an arc $\gammaup$ in an oriented configuration let $\gammaup^0$ and $\gammaup^1$ be the starting point and the end point of $\gammaup$.
\end{defn}
Recall that the odd Khovanov differential raises the h-grading by 1. Now we want to define higher differentials on the odd Khovanov chain groups that raise the h-grading by 1 or more. So we assign maps not only to the edges but to all faces of the hypercube of resolutions. While for the odd Khovanov differential we defined maps $\partial_{\C}$ for one-dimensional configurations $\C$, we now define maps for some configurations with higher dimension. To do this we define types for an active oriented configuration $\C$ in table \ref{ttypes}.
\begin{table}[htbp]
\centering
\begin{tabular}{c|p{11.5cm}}
Type&\multicolumn{1}{c}{Condition}\\
\hline
$\ta_n$&We require $n\in\n$. The configuration $\C$ has exactly 2 circles $x_1$ and $x_2$ and exactly $n$ arcs, where all arcs point from $x_1$ to $x_2$.\\
$\tb_n$&We require $n\in\n$. The configuration $\C$ has exactly $n$ circles $x_0,\dots,x_{n-1}$ and exactly $n$ arcs $\gammaup_0,\dots,\gammaup_{n-1}$, where $\gammaup_i$ points from $x_{(i-1)\textup{ mod }n}$ to $x_i$.\\
$\tc_{p,q}$&We require $p\le q\in\n$. The configuration $\C$ has exactly 1 circle $x$ and we can choose an orientation on $x$ as clockwise direction, which determines interior and exterior of $x$, so that the following holds: Our $\C$ has exactly $p$ arcs $\gammaup_1,\dots,\gammaup_p$ in the interior and exactly $q$ arcs $\delta_1,\dots,\delta_q$ in the exterior. If we start at $\gammaup_1^0$ and go along $x$ in clockwise direction we reach the boundary points of the arcs in the order $\gammaup_1^0,\dots,\gammaup_p^0,\delta_1^0,\dots,\delta_q^0,\gammaup_p^1,\dots
\gammaup_1^1,\delta_q^1,\dots\delta_1^1$.\\
$\tde_{p,q}$&We require $p\le q\in\n$. There is exactly one circle $z$ of $\C$ on which lie two starting points and two end points of arcs. We can choose an orientation on $z$ as clockwise direction, so that in the interior of $z$ lie the circles $x_1,\dots,x_{p-1}$ and in the exterior of $z$ lie the circles $y_1,\dots,y_{q-1}$. The arcs of $\C$ are $\gammaup_1,\dots,\gammaup_p,\delta_1,\dots,\delta_q$ and $\gammaup_i$ points from $x_{i-1}$ to $x_i$ for $i\in\{2,\dots,p-1\}$ while $\gammaup_1$ points from $z$ to $x_1$ and $\gammaup_p$ points from $x_{p-1}$ to $z$. Furthermore $\delta_i$ points from $y_{i-1}$ to $y_i$ for $i\in\{2,\dots,q-1\}$ while $\delta_1$ points from $z$ to $y_1$ and $\delta_q$ points from $y_{q-1}$ to $z$. If we go along $z$ starting at $\gammaup_1^0$ we reach in clockwise direction $\gammaup_1^0,\delta_1^0,\gammaup_p^1,\delta_q^1$.\\
$\tf_{p,q}$&We require $p,q\in\n_0,\ p+q\ge 1$. There is exactly one circle $y$ in $\C$ so that all arcs that start on $y$ also end on $y$. These arcs are called $\delta_1,\dots,\delta_q$. If we choose an orientation on $y$ as clockwise direction, then for all $\delta_i$ holds: If $\delta_i$ is in the interior of $y$ and we go along $y$ in clockwise direction starting at $\delta_i^0$ the first boundary point of an arc that we reach is $\delta_i^1$. If $\delta_i$ is in the exterior of $y$ the analogous statement holds with counter-clockwise instead of clockwise. Despite from the $y$-circle $\C$ has the circles $x_1,\dots,x_p$ and despite from the $\delta_i$-arcs $\C$ has the arcs $\gammaup_1,\dots,\gammaup_p$, where $\gammaup_i$ points from $x_i$ to $y$.\\
$\tg_{p,q}$&We require $p,q\in\n_0,\ p+q\ge 1$. The configuration $\textup{r}(\C)$ is of type $\tf_{p,q}$.\\
\hline
\end{tabular}
\caption{Types for an active oriented configuration $\C$}
\label{ttypes}
\end{table}
Our types are the same as the types of Szab\'o \cite[Section 4]{szabo}, except of Szab\'o's type $\te_{p,q}$, which is split into our types $\tf_{p,q}$ and $\tg_{p,q}$. Examples for the types are given by Szab\'o \cite[Figure 3]{szabo}. In Szab\'o's figure, the configuration in the middle of the lower row is of our type $\tf_{6,5}$ and the configuration in the right of the lower row is of our type $\tg_{3,4}$.
\begin{defn}
Define
\[\ttt\defeq\left\{\textup{A}_n,\ \textup{B}_n,\ \textup{C}_{p,q},\ \textup{D}_{p,q},\ \textup{F}_{r,s},\ \textup{G}_{r,s}\ \ |\ \ n,p,q\in\mathbb{N}\ ,p\le q,\ r,s\in\mathbb{N}_0,\ r+s\ge 1\right\}.\]
\end{defn}
If $\C$ is disconnected there is no $\tau\in\ttt$, so that $\C$ is of type $\tau$. If $\C$ is at least 3-dimensional there is at most one $\tau\in\ttt$ so that $\C$ is of type $\tau$. If $\C$ is 2-dimensional and of type Y (see Figure \ref{fxy}) then $\C$ is of type $\tc_{1,1}$ and $\tde_{1,1}$. If $\C$ is 2-dimensional and not of type Y then there is at most one $\tit$ so that $\C$ is of type $\tau$. If $\C$ is 1-dimensional there are two possibilities (see Figure \ref{1dim}). If $\C$ is a join, then it is of type $\ta_1$, $\tf_{1,0}$ and $\tg_{1,0}$. If $\C$ is split, then it is of type $\tb_1$, $\tf_{0,1}$ and $\tg_{0,1}$. For every $\tit$ there exists a $\C$ of type $\tau$. For $\tau=\ta_n,\tb_n,\tc_{p,q},\tde_{p,q}$ there exists exactly one $\C$ of type $\tau$. The configuration $\C$ is of type $\ta_n$, $\tc_{p,q}$, $\tf_{p,q}$ if and only if $\textup{m}(\C^*)$ is of type $\tb_n$, $\tde_{p,q}$, $\tg_{q,p}$. So for $\tit$ we define
\[\textup{m}(\tau^*)\defeq\begin{cases}
\textup{B}_n\qquad&\textup{for $\tau=\textup{A}_n$,}\\
\textup{A}_n\qquad&\textup{for $\tau=\textup{B}_n$,}\\
\textup{D}_{p,q}\qquad&\textup{for $\tau=\textup{C}_{p,q}$,}\\
\textup{C}_{p,q}\qquad&\textup{for $\tau=\textup{D}_{p,q}$,}\\
\textup{G}_{q,p}\qquad&\textup{for $\tau=\textup{F}_{p,q}$,}\\
\textup{F}_{q,p}\qquad&\textup{for $\tau=\textup{G}_{p,q}$.}
\end{cases}\]
For $\C$ an active oriented $n$-dimensional configuration of type $\tit$ and $\eps$ an edge assignment of type Y respective $\C$ we want to define a $\z$-module-homomorphism
\[\td_{\C,\tau,\eps}:\Lambda\V(\C)\to\Lambda\V(\C^*).\]
On the free abelian group $\Lambda\V(\C)$ there is, up to sign, a canonical basis. Exactly one of these basis elements should be mapped non-trivially. We will describe in the following which element this is and how it is mapped. For this we choose an \emph{allowed} path of edges $\theta$: For $\tau\ne\tde_{p,q}$ every path of edges with length $n$ (and with edges in $\f(n,1)$) is allowed, for $\tau=\tde_{p,q}$ we will describe soon which paths of edges are allowed. Then we define $x_{\C,\tau,\theta}\in\Lambda\V(\C)$ and $y_{\C,\tau}\in\Lambda\V(\C^*)$ and set
\[\td_{\C,\tau,\eps}(x_{\C,\tau,\theta})=\eps(\theta)\tsp(\C,\theta)y_{\C,\tau}.\]
We will show that this is independent of the choice of $\theta$. So let us now look at each type separately. Recall that our whole construction is an integral lift of Szab\'o's construction \cite[Section 4]{szabo}.
\begin{enumerate}
\item[$\tau=\ta_n$:]
Then we set
\begin{align*}
x_{\C,\textup{A}_n,\theta}&=1,\\
y_{\C,\textup{A}_n}&=(-1)^{n+1}.
\end{align*}
Now we show that $\td_{\C,\ta_n,\eps}$ is independent of $\theta$: So let $\zetaup$ be another path of edges with length $n$. We only have to consider the case where $\phi(\zetaup)=\phi(\theta)\circ{(i\ i+1)}$ for some $i\in\{1,\dots,n-1\}$ because $\textup{S}_n$ is generated by such transpositions. In this case $({}_1\zetaup,\dots,{}_n\zetaup)$ and $({}_1\theta,\dots,{}_n\theta)$ differ only in the $i$-th and $(i+1)$-th edge. Let $s\in\f(n,2)$ be defined as follows:
\[s_j\defeq\begin{cases}
*\qquad&\textup{for $j={}_i\overline{\theta}={}_{i+1}\overline{\zetaup}$ or $j={}_{i+1}\overline{\theta}={}_{i}\overline{\zetaup}$,}\\
{}_i\zetaup_j={}_{i+1}\zetaup_j={}_i\theta_j={}_{i+1}\theta_j\qquad&\textup{otherwise.}
\end{cases}\]
Then we have $\eps(\theta)\eps(\zetaup)=1$, if $\tr(\C,s)$ is of type A or Y, and we have ${\eps(\theta)\eps(\zetaup)=-1}$, if $\tr(\C,s)$ is of type K or X, because $\eps$ is an edge assignment of type Y. Since $\C$ is of type $\ta_n$ we have the following possibilities for $\tr(\C,s)$:
\begin{enumerate}
\item[-]
$\tr(\C,s)$ has two circles and is disconnected.
\item[-]
1. in Figure \ref{f2}
\item[-]
2. in Figure \ref{f2}
\end{enumerate}
It follows that $\eps(\theta)=\eps(\zetaup)$ and $\tsp(\C,\theta)=\tsp(\C,\zetaup)$.
\item[$\tau=\tb_n$:]
Let $\gammaup_1,\dots,\gammaup_n$ be the given numbering on the arcs of $\C$. The circles of $\C$ are named $x_1,\dots,x_n$ so that $\gammaup_i^1$ lies on $x_i$. The configuration $\C^*$ is of type~$\ta_n$. The circles of $\ta_n$ are named $y_1$, $y_2$ so that all arcs point from $y_1$ to $y_2$. Furthermore let $\theta$ be a path of edges of length $n$. Then we set
\begin{align*}
x_{\C,\textup{B}_n,\theta}&=x_{\phi(\theta)(1)}\wedge\cdots\wedge x_{\phi(\theta)(n)},\\
y_{\C,\textup{B}_n}&=y_1\wedge y_2.
\end{align*}
Let $\zetaup$ fulfill $\phi(\zetaup)=\phi(\theta)\circ{(i\ i+1)}$ for some $\in\{1,\dots,n-1\}$. We have $x_{\C,\textup{B}_n,\theta}=-x_{\C,\textup{B}_n,\zetaup}$. For $\tr(\C,s)$ there are the following possibilities:
\begin{enumerate}
\item[-]
$\tr(\C,s)$ has four circles.
\item[-]
$\tr(\C,s)$ has three circles and is connected.
\item[-]
1. in Figure \ref{f1}
\end{enumerate}
It follows that $\eps(\theta)=-\eps(\zetaup)$ and $\tsp(\C,\theta)=\tsp(\C,\zetaup)$.
\item[$\tau=\tc_{p,q}$:]
Then we set
\begin{align*}
x_{\C,\textup{C}_{p,q},\theta}&=1,\\
y_{\C,\textup{C}_{p,q}}&=(-1)^n.
\end{align*}
Let $\zetaup$ fulfill $\phi(\zetaup)=\phi(\theta)\circ{(i\ i+1)}$ for some $\in\{1,\dots,n-1\}$. For $\tr(\C,s)$ there are the following possibilities:
\begin{enumerate}
\item[-]
$\tr(\C,s)$ has three circles and is disconnected.
\item[-]
2. in Figure \ref{f1}
\item[-]
$\tr(\C,s)$ has two circles and is disconnected.
\item[-]
One of the configurations in Figure \ref{f2}
\item[-]
$\tr(\C,s)$ has type Y.
\end{enumerate}
In the first two cases we have $\eps(\theta)=-\eps(\zetaup)$ and $\tsp(\C,\theta)=-\tsp(\C,\zetaup)$. In the last three cases we have $\eps(\theta)=\eps(\zetaup)$ and $\tsp(\C,\theta)=\tsp(\C,\zetaup)$.
\item[$\tau=\tde_{p,q}$:]
Let $x$ be the circle of $\C$ on which lie two starting points and two end points of arcs. Let $\gammaup_1,\dots,\gammaup_n$ be the given numbering on the arcs of $\C$. Let $\gammaup_a$, $\gammaup_b$ be the two arcs that point to $x$. For $i\in\{1,\dots,n\}\setminus\{a,b\}$ let $x_i$ be the circle to which $\gammaup_i$ points. The circle of $C^*$ is called $y$. A path of edges $\theta$ with length $n$ (and with edges in $\f(n,1)$) is called \emph{allowed} if $\{\phi(\theta)(n-1),\ \phi(\theta)(n)\}=\{a,\ b\}$. Then we set
\begin{align*}
x_{\C,\textup{D}_{p,q},\theta}&=x_{\phi(\theta)(1)}\wedge\cdots\wedge x_{\phi(\theta)(n-2)}\wedge x,\\
y_{\C,\textup{D}_{p,q}}&=y.
\end{align*}
There are two cases to consider. First let $\zetaup$ fulfill $\phi(\zetaup)=\phi(\theta)\circ{(n-1\ n)}$. Then $\tr(\C,s)$ is of type Y. It follows that $\eps(\theta)=\eps(\zetaup)$ and $\tsp(\C,\theta)=\tsp(\C,\zetaup)$. Let $\zetaup$ now fulfill $\phi(\zetaup)=\phi(\theta)\circ{(i\ i+1)}$ for some $i\in\{1,\dots,n-3\}$. Then we have $x_{\C,\textup{D}_{p,q},\theta}=-x_{\C,\textup{D}_{p,q},\zetaup}$. For $\tr(\C,s)$ there are the following possibilities:
\begin{enumerate}
\item[-]
$\tr(\C,s)$ has four circles.
\item[-]
$\tr(\C,s)$ has three circles and is connected.
\end{enumerate}
It follows that $\eps(\theta)=-\eps(\zetaup)$ and $\tsp(\C,\theta)=\tsp(\C,\zetaup)$.
\item[$\tau=\tf_{p,q}$:]
Let $\gammaup_1,\dots,\gammaup_n$ be the given numbering on the arcs of $\C$. For $i\in\{1,\dots,n\}$ set $x_i\defeq 1\in\Lambda\V(\C)$ if $\gammaup_i^0$ and $\gammaup_i^1$ lie on the same circle. Otherwise let $x_i$ be the circle on which $\gammaup_i^0$ lies. In $\C^*$ there is exactly one circle $y$ such that all arcs which end on $y$ also start on $y$. Then we set
\begin{align*}
x_{\C,\textup{F}_{p,q},\theta}&=x_{\phi(\theta)(1)}\wedge\cdots\wedge x_{\phi(\theta)(n)},\\
y_{\C,\textup{F}_{p,q}}&=y.
\end{align*}
Let $\zetaup$ fulfill $\phi(\zetaup)=\phi(\theta)\circ{(i\ i+1)}$ for some $i\in\{1,\dots,n-1\}$. First look at the case where $x_{\phi(\theta)(i)}=1$ or $x_{\phi(\theta)(i+1)}=1$. Then we have $x_{\C,\textup{F}_{p,q},\theta}=x_{\C,\textup{F}_{p,q},\zetaup}$. For $\tr(\C,s)$ there are the following possibilities:
\begin{enumerate}
\item[-]
2. in Figure \ref{f1}
\item[-]
3. in Figure \ref{f1}
\item[-]
2. in Figure \ref{f2}
\item[-]
3. in Figure \ref{f2}
\end{enumerate}
In the first two cases we have $\eps(\theta)=-\eps(\zetaup)$ and $\tsp(\C,\theta)=-\tsp(\C,\zetaup)$. In the last two cases we have $\eps(\theta)=\eps(\zetaup)$ and $\tsp(\C,\theta)=\tsp(\C,\zetaup)$. Now look at the case where $x_{\phi(\theta)(i)}\ne 1\ne x_{\phi(\theta)(i+1)}$. Then we have $x_{\C,\textup{F}_{p,q},\theta}=-x_{\C,\textup{F}_{p,q},\zetaup}$ and $\tr(\C,s)$ is connected and has three circles. It follows that $\eps(\theta)=-\eps(\zetaup)$ and $\tsp(\C,\theta)=\tsp(\C,\zetaup)$.
\item[$\tau=\tg_{p,q}$:]
Let $\gammaup_1,\dots,\gammaup_n$ be the given numbering on the arcs of $\C$. For $i\in\{1,\dots,n\}$ set $x_i\defeq 1\in\Lambda\V(\C)$ if $\gammaup_i^0$ and $\gammaup_i^1$ lie on the same circle. Otherwise let $x_i$ be the circle on which $\gammaup_i^1$ lies. In $\C^*$ there is exactly one circle $y$ such that all arcs which start on $y$ also end on $y$. Then we set
\begin{align*}
x_{\C,\textup{G}_{p,q},\theta}&=x_{\phi(\theta)(1)}\wedge\cdots\wedge x_{\phi(\theta)(n)},\\
y_{\C,\textup{G}_{p,q}}&=(-1)^{p+1}y.
\end{align*}
That $\td_{\C,\tg_{p,q},\eps}$ is independent of $\theta$ can be shown analogously to the case $\tau=\tf_{p,q}$.
\end{enumerate}
Directly from the definition we see:
\begin{lem}[grading rule]
For $x\in\Lambda^m\V(\C)$ we write $\gr(x)=m$. Then for every allowed path of edges $\theta$ we have
\[\textup{gr}(y_{\C,\tau})-\textup{gr}(x_{\C,\tau,\theta})=\frac{|\C^*|-|\C|-n}{2}+1
=\textup{sp}(\C,(1,\dots,1))-n+1.\]
\end{lem}
Next we want to define $\td_{\C,\tau,\eps}$ for the case when $\C$ is not necessarily active. For $m\in\n_0$ and arbitrarily chosen passive circles $z_1,\dots,z_m$ of $\C$ set $\omega\defeq z_1\wedge\cdots\wedge z_m$. Furthermore let $\theta$ be a path of edges that is allowed for $\textup{act}(\C)$. Then we call $\theta$ \emph{allowed for $\C$}. Now let
\[\td_{\C,\tau,\eps}(x_{\textup{act}(\C),\tau,\theta}\wedge\omega)
=\eps(\theta)\textup{sp}(\C,\theta)y_{\textup{act}(\C),\tau}\wedge\omega.\]
All canonical basis elements of $\Lambda\V(\C)$ that cannot be written in the form $x_{\textup{act}(\C),\tau,\theta}\wedge\omega$, with $\omega$ as described, are mapped trivially.
\begin{defn}
We set
\[\td_{\C,\eps}\defeq\sum_{\tau\in\textup{T},\textup{ $\C$ is of type $\tau$}}\td_{\C,\tau,\eps}.\]
\end{defn}
One can easily see:
\begin{lem}\label{lemd1}
If $\C$ is 1-dimensional, then $\td_{\C,\eps}$ equals the odd Khovanov differential $\partial(\C,\eps)=\eps(*)\partial_\C$.
\end{lem}
Before we can define the higher differentials on the odd Khovanov chain groups we have to describe what the restriction of an edge assignment should be.
\begin{defn}
Let $\eps:\f(n,1)\to\{-1,1\}$ be an edge assignment. For all $a\in\f(n,k)$, $1\le k\le n$, our $\eps$ induces an edge assignment $\eps|_a:\f(k,1)\to\{-1,1\}$ as follows: Let $i_1<i_2<\dots<i_k$ be so that $a_{i_j}=*$ for all $j\in\{1,\dots,k\}$. For $\lambda\in\f(k,1)$ let $\mu\in\f(n,1)$ with $\mu_l\defeq a_l$ if $a_l\ne*$ and $\mu_{i_j}\defeq\lambda_j$ for all $j\in\{1,\dots,k\}$. Then let $\eps|_a(\lambda)\defeq\eps(\mu)$.
\end{defn}
Now let $\C$ be an $n$-dimensional oriented configuration, $\eps$ an edge assignment of type Y with respect to $\C$ and let $k\in\{1,\dots,n\}$. For all $a\in\f(n,k)$ we get the map
\[\td_{\tr(\C,a),\eps|_a}:\Lambda\underbrace{\V(\tr(\C,a))}_{=\V(\tr(\C,a^0))}\to
\Lambda\underbrace{\V((\tr(\C,a))^*)}_{=\V(\tr(\C,a^1))}.\]
This allows us to define the higher differentials:
\begin{defn}
Set
\[\td_k(\C,\eps)\defeq\bigoplus_{a\in\f(n,k)}
(-1)^{|a^0|+(k+1)\textup{sp}(\C,a^0)}
\td_{\tr(\C,a),\eps|_a}:\textup{C}(\C)\to\textup{C}(\C)\]
and $\td(\C,\eps)\defeq\sum_{1\le k\le n}\td_k(\C,\eps)$.
\end{defn}
From the grading rule we get:
\begin{cor}\label{corgr}
The map $\td_k(\C,\eps)$ raises the h-grading by $k$ and decreases the $\delta$-grading by 2.
\end{cor}
Now we state that $\td(\C,\eps)$ is really a differential:
\begin{thm}\label{thdd}
We have $\td(\C,\eps)\circ\td(\C,\eps)=0$.
\end{thm}
We sketch the proof \cite{master} in Section \ref{sdd}.
\begin{defn}
Let $\hc(\C,\eps)$ be the chain complex $\left(\tc(\C),\td(\C,\eps)\right)$.
\end{defn}
From Lemma \ref{lemd1} we get:
\begin{cor}\label{cord1}
The chain complex $\left(\tc(\C),\td_1(\C,\eps)\right)$ is isomorphic to the odd Khovanov complex $\textup{C}(\C,\eps)=\left(\tc(\C),\partial(\C,\eps)\right)$.
\end{cor}
By comparing the construction of Szab\'o \cite{szabo} with our construction we easily see that we have constructed a lift of Szab\'o's complex:
\begin{lem}
The complex $\hc(\C,\eps)\otimes_\z\z_2$ is isomorphic to the complex constructed by Szab\'o.
\end{lem}
Next we state that $\hc(\C,\eps)$ is independent of $\eps$ and the orientation on the arcs of $\C$.
\begin{thm}\label{thind}
Let $\C$ and $\D$ be oriented configurations with $\overline{\C}=\overline{\D}$. Let $\eps$ and $\etaup$ be edge assignments so that $\eps$ with respect to $\C$ and $\etaup$ with respect to $\D$ has type Y. Then we have $\hc(\C,\eps)\cong\hc(\D,\etaup)$.
\end{thm}
We will prove this in Section \ref{sdep}. To get our spectral sequence we need:
\begin{defn}
With the h-grading we get a filtration $\tf$ on $\hc(\C,\eps)$ by $\tf_i\tc(\C)\defeq\oplus_{h\ge i}\tc(\C)^h$.
\end{defn}
Due to the theory of spectral sequences \cite{weibel} we can define:
\begin{defn}
Let $\te(\C,\eps)$ be the spectral sequence we get from the filtration $\tf$ on $\hc(\C,\eps)$.
\end{defn}
From Corollary \ref{cord1} we get:
\begin{cor}
We have a canonical isomorphism from $\te^2(\C,\eps)$ to the odd Khovanov homology $\textup{H}(\tc(\C,\eps))$.
\end{cor}
The theory of spectral sequences gives:
\begin{lem}
The spectral sequence $\te(\C,\eps)$ converges to the homology group $\textup{H}(\hc(\C,\eps))$, which has a grading induced by the $\delta$-grading on $\tc(\C)$ because of Corollary \ref{corgr}.
\end{lem}
Next we define the spectral sequence for a link diagram $\D$.
\begin{defn}
Let $\C$ be an oriented configuration so that $\overline{\C}$ is a 0-resolution of $\D$. Furthermore let $\eps$ be an edge assignment of type Y respective $\C$. Then $\hc(\D,\C,\eps)$ is the bigraded chain complex that only differs from $\hc(\C,\eps)$ in the bigrading as described in Definition \ref{defkd}. From $\hc(\D,\C,\eps)$ we get the spectral sequence $\te(\D,\C,\eps)$
\end{defn}
The spectral sequence is Reidemeister invariant in the following sense.
\begin{thm}
Let $\D_1$ and $\D_2$ be link diagrams that differ by finitely many Reidemeister moves. Then there exist oriented configurations $\C_1$, $\C_2$, so that $\overline{\C}_1$, $\overline{\C}_2$ are 0-resolutions of $\D_1$, $\D_2$ and edge assignments $\eps_1$, $\eps_2$, so that $\eps_1$ with respect to $\C_1$ and $\eps_2$ with respect to $\C_2$ has type Y and so that $\te(\D_1,\C_1,\eps_1)\cong\te(\D_2,\C_2,\eps_2)$ and for the homology groups we have $\textup{H}_{\delta}(\hc(\D_1,\C_1,\eps_1))
\cong\textup{H}_{\delta}(\hc(\D_2,\C_2,\eps_2))$.
\end{thm}
\begin{proof}
This follows from the proofs of the Reidemeister invariance of Szab\'o's spectral sequence \cite{szabo} and of odd Khovanov homology \cite{ors}.
\end{proof}
So the spectral sequence and the homology of our complex are link invariants and we have proved Theorem \ref{mainth}.
\begin{defn}
For a link $L$ we denote by $\te(L)$ our spectral sequence. Furthermore $\hh(L)$ is the graded abelian group that arises as the homology of our complex.
\end{defn}
Now $\te(L)$ and $\hh(L)$ have the properties stated in Corollary \ref{maincor}: The $\te^2$-term of our spectral sequence is isomorphic to odd Khovanov homology $\te^2(L)\cong\textup{Kh}'(L)$. The spectral sequence $\te(L)$ converges to $\hh(L)$, from which one can get Szab\'o's $\hh(L;\z_2)$ \cite{szabo} with the Universal Coefficient Theorem.

We have defined $\hc(\C,\eps)$ only if $\eps$ has type Y with respect to $\C$. For type X we do the following:
\begin{defn}
Let $\eps$ be an edge assignment of type X with respect to $\C$. Then $\eps$ is of type Y with respect to $\textup{r}(\textup{m}(\C))$. So we can define
\[\hc '(\C,\eps)\defeq\left(\tc '(\C),\td '(\C,\eps)\right)\defeq\hc(\textup{r}(\textup{m}(\C)),\eps).\]
\end{defn}
From Corollary \ref{cord1} we get:
\begin{cor}
The chain complex $\left(\tc '(\C),\td_1'(\C,\eps)\right)$ is isomorphic to the odd Khovanov complex $\textup{C}(\textup{r}(\textup{m}(\C)),\eps)\cong\textup{C}(\C,\eps)$.
\end{cor}
Szab\'o gives a second version of his complex with differential $\td '$ \cite[beginning of Section 8]{szabo}. Obviously we have constructed a lift of this complex:
\begin{lem}
The complex $\hc '(\C,\eps)\otimes_\z\z_2$ is isomorphic to Szab\'o's complex with differential $\td '$.
\end{lem}
So Proposition \ref{primelem} is proved.
\begin{defn}
Let $\te '(\C,\eps)$ be the spectral sequence we get from $\hc '(\C,\eps)$.
\end{defn}
As in the proof of Corollary \ref{xycor2} we conclude from the Reidemeister invariance:
\begin{cor}
Let $\eps$, $\etaup$ be edge assignments of type X,Y with respect to $\C$. Then we have $\te '(\C,\eps)\cong\te(\C,\etaup)$ and $\textup{H}_\delta\bigl(\hc '(\C,\eps)\bigr)\cong\textup{H}_\delta\bigl(\hc(\C,\etaup)\bigr)$.
\end{cor}
This proves Corollary \ref{primecor}. The rest of this article is dedicated to the proofs of Theorem \ref{thdd} and Theorem \ref{thind}.

\section{Dependence on the orientation}\label{sdep}
In this Section we analyze how our differential depends on the orientation of the arcs. This investigation is obviously necessary for the proof of Theorem~\ref{thind} but it also helps us to prove Theorem \ref{thdd}. First we give two relations that the differential obviously fulfills, the \emph{filtration rule} and the \emph{duality rule}. For this we have to fix some notation. In the following let $\C$ be an orientated configuration.
\begin{defn}
For $x_1,\dots,x_n$ mutually distinct circles of $\C$ we call $x_1\wedge\cdots\wedge x_n\in\Lambda\V(\C)$ a \emph{monomial of degree $n$ which is divisible by $x_1,\dots,x_n$}.
\end{defn}
\begin{defn}
Let $\alpha\in\Lambda\V(\C)$ be a monomial and let $\omega\in\Lambda\V(\C)$. Furthermore let $\beta_1,\dots,\beta_m\in\Lambda\V(\C)$ be monomials such that $\{\alpha,\beta_1,\dots,\beta_m\}$ is a basis of the free abelian group $\Lambda\V(\C)$. There are $\lambda,\mu_1,\dots,\mu_m\in\z$ such that $\omega=\lambda\alpha+\sum_{i=1}^m\mu_i\beta_i$. Then $\lambda$ is called the \emph{coefficient of $\omega$ at $\alpha$}. This does not depend on the choice of the $\beta_i$.
\end{defn}
\begin{defn}
Let $P$ be a point on a circle of $\C$ which does not lie on an arc. Then $x(P)$ and $y(P)$ denote the starting circle and the ending circle of $\C$ on which $P$ lies.
\end{defn}
Equipped with these definitions we can now state the filtration rule and the duality rule. Let $P$ be a point on a circle of $\C$ which does not lie on an arc. Let $\C$ be of type $\tit$ and $\eps$ be an edge assignment of type Y with respect to $\C$. Let $\alpha\in\Lambda\V(\C)$ and $\beta\in\Lambda\V(\C^*)$ be monomials. Directly from the definition of the differential we see:
\begin{lem}[filtration rule]
If $\alpha$ is divisible by $x(P)$ and the coefficient of $\td_{\C,\tau,\eps}(\alpha)$ at $\beta$ is non-trivial, then $\beta$ is divisible by $y(P)$.
\end{lem}
Let $\alpha^*\in\Lambda\V(\textup{m}(\C))$ be a monomial with $\textup{gr}(\alpha^*)=|\C|-\textup{gr}(\alpha)$ such that $\alpha^*$ is divisible by a circle $x$ in $\textup{m}(\C)$ if and only if $\alpha$ is not divisible by the counterpart of $x$ in $\C$. Then $\alpha^*$ is unique up to sign. Let $\beta^*\in\Lambda\V(\textup{m}(\C^*))$ be a monomial which is related to $\beta$ as $\alpha^*$ is related to $\alpha$. Let $\etaup$ be an edge assignment of type Y with respect to $\textup{m}(\C^*)$. Directly from the definition of the differential we see:
\begin{lem}[duality rule]
The coefficient of $\td_{\C,\tau,\eps}(\alpha)$ at $\beta$ equals up to sign the coefficient of $\td_{\textup{m}(\C^*),\textup{m}(\tau^*),\etaup}(\beta^*)$ at $\alpha^*$.
\end{lem}
The filtration rule also holds for
\[\td_{\C,\eps}=\sum_{\tau\in\textup{T},\textup{ $\C$ is of type $\tau$}}\td_{\C,\tau,\eps}.\] The duality rule also holds for $\td_{\C,\eps}$ and $\td_{\textup{m}(\C^*),\etaup}$. To compare the differentials of two configurations that only differ in the orientation of the arcs we need to define "edge-homotopy" maps:
\begin{defn}
Let $\C$ be an active oriented one-dimensional configuration. We define a $\z$-module-homomorphism $\tth_\C:\Lambda\V(\C)\to\Lambda\V(\C^*)$:
\begin{enumerate}
\item[-]
If $\C$ is a join (see Figure \ref{1dim}) then $\tth_\C(1)=\tth_\C(x_1)=\tth_\C(x_2)=0$ and $\tth_\C(x_1\wedge x_2)=y$.
\item[-]
If $\C$ is a split (see Figure \ref{1dim}) then $\tth_\C(1)=1$ and $\tth_\C(y)=0$.
\end{enumerate}
If $\C$ is not active then we define $\tth_\C$ analogously to $\partial_\C$ (Definition \ref{nact}).
\end{defn}
The map $\tth_\C$ fulfills the filtration rule and the duality rule.
\begin{defn}
Let $\C$ be an oriented $n$-dimensional configuration and $\eps$ an edge assignment of type Y with respect to $\C$. For $i\in\{1,\dots,n\}$ let $\tth_i(\C,\eps):\tc(\C)\to\tc(\C)$ be defined as
\[\bigoplus_{a\in\f(n,1)\textup{ with }a_i=*}\eps(a)(-1)^{\textup{sp}(\C,a^0)}\tth_{\tr(\C,a)}.\]
\end{defn}
We have $\tth_i(\C,\eps)\circ\tth_i(\C,\eps)=0$ because
\[\textup{im}(\tth_i(\C,\eps))\subset\bigoplus_{a\in\f(n,0)\textup{ with }a_i=1}\Lambda\textup{V}(\tr(\C,a))\subset\textup{ker}(\tth_i(\C,\eps)).\]
Furthermore $\tth_i(\C,\eps)$ raises the h-grading by 1 and does not change the $\delta$-grading. Now we can state how the differential depends on the orientation of the arcs.
\begin{thm}\label{dhth}
Let $\C,\D$ be oriented $n$-dimensional configurations such that $\overline{\C}=\overline{\D}$ and the orientation of the arcs differs only at the $i$-th arc for some $i\in\{1,\dots,n\}$. Then there is an edge assignment $\eps$ of type Y with respect to $\C$ and an edge assignment $\etaup$ of type Y with respect to $\D$ so that
\[\td(\C,\eps)-\td(\D,\etaup)=\td(\C,\eps)\circ\tth_i(\C,\eps)-\tth_i(\C,\eps)\circ \td(\C,\eps).\]
\end{thm}
\begin{proof}
It is enough to consider the case $i=1$. Let $\kappa:\f(n,1)\to\{1,-1\}$ be defined as
\[\kappa(a)=\begin{cases}
-1\qquad&\textup{if $a_1=*$ and $\tr(\C,a)$ is a split,}\\
1\qquad&\textup{otherwise.}
\end{cases}\]
Then $\etaup\defeq\kappa\eps$ is an edge assignment of type Y with respect to $\D$ \cite[Proof of Lemma 2.3]{ors}. Furthermore we have $\td_1(\C,\eps)=\td_1(\D,\etaup)$. Now let $n\ge 2$ and $a=(0,*,\dots,*)$, $b=(1,*,\dots,*)\in\f(n,n-1)$ and $c=(*,0,\dots,0)$, $d=(*,1,\dots,1)\in\f(n,1)$. We have to show:
\begin{equation}\label{dheq}
\td_{\C,\eps}-\td_{\D,\etaup}=(-1)^{1+n\textup{sp}(\C,b^0)}\eps(c)\td_{\tr(\C,b),\eps|_b}\circ H_{\tr(\C,c)}-(-1)^{\textup{sp}(\C,d^0)}\eps(d)\tth_{\tr(\C,d)}\circ \td_{\tr(\C,a),\eps|_a}.
\end{equation}
We only have to consider the case where $\C$ is active. 

First let $\C$ be disconnected. Then $\td_{\C,\eps}=0=\td_{\D,\etaup}$. We only have to consider the case where $\C$ has exactly two connected components such that in one connected component there lies the first arc and no other arc because otherwise $\td_{\tr(\C,b),\eps|_b}=0=\td_{\tr(\C,a),\eps|_a}$. Let $\theta$ be a path of edges that is allowed for $\textup{act}(\tr(\C,a))=\textup{act}(\tr(\C,b))$. We have to consider two cases:
\begin{enumerate}
\item[-]
Let $\textup{act}(\tr(\C,c))=\textup{act}(\tr(\C,d))$ be a join (see Figure \ref{1dim}). Then we have $\tsp(\C,b^0)=0$. Moreover $\eps(d)\eps|_a(\theta)=(-1)^{n-1}\eps(c)\eps|_b(\theta)$ because $\tr(\C,s)$ is of type K for all $s\in\f(n,2)$ with $s_1=*$. If $\tr(\C,a)$ is of type $\tit$ we have:
\begin{align*}
&\td_{\tr(\C,b),\tau,\eps|_b}\circ\tth_{\tr(\C,c)}(x_1\wedge x_2\wedge x_{\tr(\C,b),\tau,\theta})\\
=&\td_{\tr(\C,b),\tau,\eps|_b}(y\wedge x_{\tr(\C,b),\tau,\theta})\\
=&(-1)^{\textup{gr}(x_{\tr(\C,b),\tau,\theta})}\eps|_b(\theta)
\textup{sp}(\tr(\C,b),\theta)y_{\tr(\C,b),\tau}\wedge y
\end{align*}
and
\begin{align*}
&\tth_{\tr(\C,d)}\circ\td_{\tr(\C,a),\tau,\eps|_a}(x_1\wedge x_2\wedge x_{\tr(\C,b),\tau,\theta})\\
=&\eps|_a(\theta)\textup{sp}(\tr(\C,a),\theta)\tth_{\tr(\C,d)}(y_{\tr(\C,a),\tau}\wedge x_1\wedge x_2)\\
=&\eps|_a(\theta)\textup{sp}(\tr(\C,b),\theta)y\wedge y_{\tr(\C,b),\tau}\\
=&(-1)^{\textup{gr}(y_{\tr(\C,b),\tau})}\eps|_a(\theta)\textup{sp}(\tr(\C,b),\theta) y_{\tr(\C,b),\tau}\wedge y.
\end{align*}
Because of
\begin{align*}
\textup{gr}(y_{\tr(\C,b),\tau})-\textup{gr}(x_{\tr(\C,b),\tau,\theta})
&=\textup{sp}(\tr(\C,b),(1,\dots,1))-(n-1)+1\\
&=\textup{sp}(\C,d^0)-n+2
\end{align*}
the right hand side of equation ($\ref{dheq}$) is trivial.
\item[-]
Let $\textup{act}(\tr(\C,c))=\textup{act}(\tr(\C,d))$ be a split (see Figure \ref{1dim}). Then we have $\tsp(\C,b^0)=1$. Furthermore $\eps(d)\eps|_a(\theta)=(-1)^{n-1-\textup{sp}(\C,d^0)}\eps(c)\eps|_b(\theta)$ because of the following. Let $s\in\f(n,2)$ and $e\in\f(n,1)$ be equal despite of the first component. Let $s_1=*$ and $e_1=0$. Then $\tr(\C,s)$ is of type K if $\tr(\C,e)$ is a join and $\tr(\C,s)$ is of type A if $\tr(\C,e)$ is a split. So the right hand side of equation~(\ref{dheq}) is trivial.
\end{enumerate}

Now let $\C$ be connected.

First let $\C$ be as in Figure \ref{f3}.
\begin{figure}[htbp]
\centering
\includegraphics{gra.6}
\caption{}
\label{f3}
\end{figure}
Because $\C$ is of type~A and $\D$ is of type~K we have
\[\eps(0,*)\eps(*,1)=\eps(*,0)\eps(1,*)=\etaup(*,0)\etaup(1,*)=-\etaup(0,*)\etaup(*,1).\]
Moreover $\tsp(\C,(1,0))=\tsp(\C,(0,1))=0$. We have $\td_{\tc,\eps}(1)=-\eps(0,*)\eps(*,1)$ and $\td_{\D,\etaup}(x_1\wedge x_2)=\etaup(*,0)\etaup(1,*)y_2\wedge y_1$. Furthermore $\td_{\tr(\C,(1,*)),\eps|_{(1,*)}}\circ\tth_{\tr(\C,(*,0))}(x_1\wedge x_2)=\eps(1,*)y_2\wedge y_1$ and $\tth_{\tr(\C,(*,1))}\circ \td_{\tr(\C,(0,*)),\eps|_{(0,*)}}(1)=\eps(0,*)$. So equation (\ref{dheq}) holds. If $\C$ is the first configuration in Figure \ref{f1}, then the roles of $\C$ and $\D$ from Figure \ref{f3} are changed and so are the roles of $\eps$ and $\etaup$. So on both sides of equation (\ref{dheq}) we get the negative of the case of Figure \ref{f3}.

Now let $\C$ be as in Figure $\ref{f4}$.
\begin{figure}[htbp]
\centering
\includegraphics{gra.7}
\caption{}
\label{f4}
\end{figure}
Because $\C$ is of type Y we have $\eps(0,*)\eps(*,1)=\eps(*,0)\eps(1,*)$. Moreover $\textup{sp}(\C,(1,0))=\textup{sp}(\C,(0,1))=1$. We have $\td_{\C,\eps}(1)=-\eps(0,*)\eps(*,1)$ and $\td_{\C,\eps}(x)=-\eps(0,*)\eps(*,1)y$ and $\td_{\D,\etaup}=0$. Furthermore\\
$\td_{\tr(\C,(1,*)),\eps|_{(1,*)}}\circ\tth_{\tr(\C,(*,0))}(1)=\eps(1,*)$ and $\tth_{\tr(\C,(*,1))}\circ\td_{\tr(\C,(0,*)),\eps|_{(0,*)}}(x)=-\eps(0,*)y$. So equation~(\ref{dheq}) holds. If $\C$ is of type X, then the roles of $\C$ and $\D$ from Figure~\ref{f4} are changed and so are the roles of $\eps$ and $\etaup$. So on both sides of equation~(\ref{dheq}) we get the negative of the case of Figure \ref{f4}.

Our $\C$ is a connected, oriented at least 2-dimensional configuration. It follows that $\C$ is the first configuration in Figure \ref{f1} or the first configuration in Figure \ref{f2} or of type X or of type Y or
\[\textup{u}_\C\defeq\#\{\tau\in\textup{T}\ \ |\ \ \textup{$\C$ or $\D$ is of type $\tau$}\}\le 1.\]

Let $\textup{u}_\C=1$. It is enough to consider the case where there is $\tit$ such that $\C$ is of type $\tau$. We look at each type separately:
\begin{enumerate}
\item[$\tau=\ta_n$, $n\ge 3$:]
Then $\tr(\C,b)$ is of no type in $\ttt$. Moreover $\tr(\C,a)$ is of type $\ta_{n-1}$ and $\tr(\C,d)$ is a split. So equation (\ref{dheq}) holds.
\item[$\tau=\tb_n$, $n\ge 3$:]
Then $\tr(\C,a)$ is of no type in $\ttt$. Moreover $\tr(\C,b)$ is of type~$\tb_{n-1}$ and $\tr(\C,c)$ is a join. So we have $\tsp(\C,b^0)=0$. Choose a numbering of the arcs of $\C$ that begins with the first arc (from the given numbering of the arcs) and on the second position we have the arc whose end point lies on the circle on which the starting point of the first arc lies. This yields allowed paths of edges $\theta$ and $\zetaup$ for $\C$ and $\tr(\C,b)$ respectively. We have $\eps(\theta)=\eps(c)\eps|_b(\zetaup)$ and $\tsp(\C,\theta)=\tsp(\tr(\C,b),\zetaup)$. Furthermore $\tth_{\tr(\C,c)}(x_{\C,\textup{B}_n,\theta})=-x_{\tr(\C,b),\textup{B}_{n-1},\zetaup}$ and $y_{\tr(\C,b),\textup{B}_{n-1}}=y_{\C,\textup{B}_n}$. So equation (\ref{dheq}) holds.
\item[$\tau=\tc_{p,q}$, $p+q\ge 3$:]
The sphere $\s^2$ without the circle $x$ of $\C$ consists of two connected components. We have to consider two cases:
\begin{enumerate}
\item[-]
First consider the case where the first arc of $\C$ is the only arc in its connected component of $\s^2\setminus x$. Then $\tr(\C,a)$ is of no type in $\ttt$. Moreover $\tr(\C,b)$ is of type $\ta_{n-1}$ and $\tr(\C,c)$ is a split. So equation (\ref{dheq}) holds.
\item[-]
Now consider the case where in the connected component of $\s^2\setminus x$ in which the first arc lies there lies at least one other arc. Then $\tr(\C,b)$ is of no type in~$\ttt$ or $\tr(\C,b)$ is of type~$\tf_{1,1}$ or of type~$\tg_{1,1}$. Because $\tr(\C,c)$ is a split we have $\td_{\tr(\C,b),\eps|_b}\circ\tth_{\tr(\C,c)}=0$. Furthermore $\tr(\C,a)$ is of type $\tc_{p-1,q}$ or $\tc_{p,q-1}$ and $\tr(\C,d)$ is a split. So equation (\ref{dheq}) holds.
\end{enumerate}
\item[$\tau=\tde_{p,q}$, $p+q\ge 3$:]
Let $x$ be the circle of $C$ on which two starting points and two end points of arcs lies. We have to consider two cases:
\begin{enumerate}
\item[-]
First consider the case where the first arc of $\C$ is the only arc in its connected component of $\s^2$. Then $\tr(\C,b)$ is of no type in $\ttt$. Moreover $\tr(\C,a)$ is of type $\tb_{n-1}$ and $\tr(\C,d)$ is a join. An allowed path of edges $\theta$ for $\C$ such that $\phi(\theta)(n)=1$ gives an allowed path of edges $\zetaup$ for $\tr(\C,a)$. We have $\eps(\theta)=\eps|_a(\zetaup)\eps(d)$ and $\tsp(\C,\theta)=(-1)^{\tsp(\C,d^0)}\tsp(\tr(\C,a),\zetaup)$. Furthermore $x_{\C,\tde_{p,q},\theta}=x_{\tr(\C,a),\tb_{n-1},\zetaup}$ and $\tth_{\tr(\C,d)}(y_{\tr(\C,a),\textup{B}_{n-1}})=-y_{\C,\textup{D}_{p,q}}$. So equation (\ref{dheq}) holds.
\item[-]
Now consider the case where in the connected component of $\s^2\setminus x$ in which the first arc lies there lies at least one other arc. Then $\tr(\C,a)$ is of no type in $\ttt$ or $\tr(\C,a)$ is of type $\tf_{1,1}$ or of type $\tg_{1,1}$. Because $\tr(\C,d)$ is a join we have $\tth_{\tr(\C,d)}\circ\td_{\tr(\C,a),\eps|_a}=0$. Moreover $\tr(\C,b)$ is of type $\tde_{p-1,q}$ or $\tde_{p,q-1}$ and $\tr(\C,c)$ is a join. It follows $\tsp(\C,b^0)=0$. There are two possibilities:
\begin{enumerate}
\item[1.]
First consider the case where the end point of the first arc does not lie on $x$. Then let $\theta$ be an allowed path of edges for $\C$ so that $\phi(\theta)(1)=1$. Now $\theta$ provides an allowed path of edges $\zetaup$ for $\tr(\C,b)$. We have $\eps(\theta)=\eps(c)\eps|_b(\zetaup)$ and $\tsp(\C,\theta)=\tsp(\tr(\C,b),\zetaup)$. Furthermore $\tth_{\tr(\C,c)}(x_{\C,\textup{D}_{p,q},\theta})=-x_{\tr(\C,b),\textup{D}_{p-1,q},\zetaup}$ and $y_{\tr(\C,b),\textup{D}_{p-1,q}}=y_{\C,\textup{D}_{p,q}}$, where instead of $\tde_{p-1,q}$ we could also have $\tde_{p,q-1}$. So equation (\ref{dheq}) holds.
\item[2.]
Now let the end point of the first arc lie on $x$. Then let $\theta$ be an allowed path of edges for $\C$ such that $\phi(\theta)(n-1)=1$ and $\phi(\theta)(n-2)$ is the number of the arc whose end point lies on the circle on which the starting point of the first arc lies. Now $\theta$ yields an allowed path of edges $\zetaup$ for $\tr(\C,b)$. We have $\eps(\theta)=(-1)^{n-2}\eps(c)\eps|_b(\zetaup)$ because $\tr(\C,s)$ is of type K for all $s\in\f(n,2)$ with $|s^1|<n$. Moreover $\tsp(\C,\theta)=\tsp(\tr(\C,b),\zetaup)$ holds. We have $\tth_{\tr(\C,c)}(x_{\C,\textup{D}_{p,q},\theta})
=(-1)^{n-3}x_{\tr(\C,b),\textup{D}_{p-1,q},\zetaup}$ and $y_{\tr(\C,b),\textup{D}_{p-1,q}}=y_{\C,\textup{D}_{p,q}}$ where instead of $\tde_{p-1,q}$ we could also have $\tde_{p,q-1}$. So equation (\ref{dheq}) holds.
\end{enumerate}
\end{enumerate}
\item[$\tau=\tf_{0,2}$:]
See Figure \ref{ff02}.
\begin{figure}[htbp]
\centering
\includegraphics{gra.8}
\caption{}
\label{ff02}
\end{figure}
Because $\C$ is of type A we have $\eps(0,*)\eps(*,1)=\eps(*,0)\eps(1,*)$. Furthermore $\textup{sp}(\C,(1,0))=\textup{sp}(\C,(0,1))=1$ holds. We have $\td_{\C,\eps}(1)=-\eps(0,*)\eps(*,1)y_2$. Moreover $\td_{\tr(\C,(1,*)),\eps|_{(1,*)}}\circ\tth_{\tr(\C,(*,0))}(1)=\eps(1,*)(y_2-y_1)$ and $\tth_{\tr(\C,(*,1))}\circ\td_{\tr(\C,(0,*)),\eps|_{(0,*)}}(1)=-\eps(0,*)y_1$. So equation (\ref{dheq}) holds.
\item[$\tau=\tf_{2,0}$:]
See Figure \ref{ff20}.
\begin{figure}[htbp]
\centering
\includegraphics{gra.9}
\caption{}
\label{ff20}
\end{figure}
Because $\C$ is of type K we have $\eps(0,*)\eps(*,1)=-\eps(*,0)\eps(1,*)$. Furthermore $\textup{sp}(\C,(1,0))=\textup{sp}(\C,(0,1))=0$. We have $\td_{\C,\eps}(x_1\wedge x_3)=\eps(0,*)\eps(*,1)y$. Moreover $\td_{\tr(\C,(1,*)),\eps|_{(1,*)}}\circ\tth_{\tr(\C,(*,0))}(x_3\wedge x_2)=\eps(1,*)y$ and $\tth_{\tr(\C,(*,1))}\circ \td_{\tr(\C,(0,*)),\eps|_{(0,*)}}(x_1\wedge x_3)=-\eps(0,*)y=\tth_{\tr(\C,(*,1))}\circ\td_{\tr(\C,(0,*)),\eps|_{(0,*)}}(x_2\wedge x_3)$. So equation (\ref{dheq}) holds.
\item[$\tau=\tf_{1,1}$:]
We have to consider two cases:
\begin{enumerate}
\item[-]
Let the two boundary points of the first arc of $\C$ lie on the same circle, see Figure \ref{ff11a}.
\begin{figure}[htbp]
\centering
\includegraphics{gra.10}
\caption{}
\label{ff11a}
\end{figure}
Because $\C$ is of type K we have $\eps(0,*)\eps(*,1)=-\eps(*,0)\eps(1,*)$. Furthermore $\textup{sp}(\C,(1,0))=1$ and $\textup{sp}(\C,(0,1))=0$. We have $\td_{\C,\eps}(x_2)=\eps(0,*)\eps(*,1)y_1$. Moreover $\td_{\tr(\C,(1,*)),\eps|_{(1,*)}}\circ\tth_{\tr(\C,(*,0))}(1)=\eps(1,*)$ and $\td_{\tr(\C,(1,*)),\eps|_{(1,*)}}\circ\tth_{\tr(\C,(*,0))}(x_2)=\eps(1,*)y_1$ and $\tth_{\tr(\C,(*,1))}\circ\td_{\tr(\C,(0,*)),\eps|_{(0,*)}}(1)=\eps(0,*)$. So equation (\ref{dheq}) holds.
\item[-]
Let the two boundary points of the first arc of $\C$ lie on distinct circles, see Figure \ref{ff11b}.
\begin{figure}[htbp]
\centering
\includegraphics{gra.11}
\caption{}
\label{ff11b}
\end{figure}
We have $\textup{sp}(\C,(1,0))=0$ and $\textup{sp}(\C,(0,1))=1$ and $\td_{\C,\eps}(x_2)=\eps(*,0)\eps(1,*)y_1$. Furthermore $\td_{\tr(\C,(1,*)),\eps|_{(1,*)}}\circ\tth_{\tr(\C,(*,0))}(x_2\wedge x_1)=\eps(1,*)y_1\wedge y_2$ and $\tth_{\tr(\C,(*,1))}\circ \td_{\tr(\C,(0,*)),\eps|_{(0,*)}}(x_2)=-\eps(0,*)y_1$ and $\tth_{\tr(\C,(*,1))}\circ\td_{\tr(\C,(0,*)),\eps|_{(0,*)}}(x_1\wedge x_2)=\eps(0,*)y_1\wedge y_2$. So equation (\ref{dheq}) holds.
\end{enumerate}
\item[$\tau=\tf_{p,q}$, $p+q\ge 3$:]
We have to consider two cases:
\begin{enumerate}
\item[-]
First consider the case where the two boundary points of the first arc of $\C$ lie on the same circle. Then $\tr(\C,a)$ and $\tr(\C,b)$ are of type $\tf_{p,q-1}$. Moreover $\tr(\C,c)$ and $\tr(\C,d)$ are splits. This implies $\tth_{\tr(\C,d)}\circ \td_{\tr(\C,a),\eps|_a}=0$ and $\textup{sp}(\C,b^0)=1$. An allowed path of edges $\theta$ for $\C$ such that $\phi(\theta)(1)=1$ gives an allowed path of edges $\zetaup$ for $\tr(\C,b)$. We have $\eps(\theta)=\eps(c)\eps|_b(\zetaup)$ and $\textup{sp}(\C,\theta)=(-1)^{1+(n-2)}\textup{sp}(\tr(\C,b),\zetaup)$. Furthermore $\tth_{\tr(\C,c)}(x_{\C,\textup{F}_{p,q},\theta})
=x_{\tr(\C,b),\textup{F}_{p,q-1},\zetaup}$ and $y_{\tr(\C,b),\textup{F}_{p,q-1}}=y_{\C,\textup{F}_{p,q}}$. So equation (\ref{dheq}) holds.
\item[-]
Now consider the case where the two boundary points of the first arc of $\C$ lie on distinct circles. Then $\tr(\C,a)$ and $\tr(\C,b)$ are of type $\tf_{p-1,q}$. Moreover $\tr(\C,c)$ and $\tr(\C,d)$ are joins. This implies $\td_{\tr(\C,b),\eps|_b}\circ\tth_{\tr(\C,c)}=0$ and $\textup{sp}(\C,d^0)=q$. An allowed path of edges $\theta$ for $\C$ such that $\phi(\theta)(n)=1$ provides an allowed path of edges $\zetaup$ for $\tr(\C,a)$. We have $\eps(\theta)=\eps|_a(\zetaup)\eps(d)$ and $\textup{sp}(\C,\theta)=(-1)^q\textup{sp}(\tr(\C,a),\zetaup)$. Let $x$ be the circle of $\C$ on which the starting point of the first arc lies. Then we have $x_{\C,\textup{F}_{p,q},\theta}=x_{\tr(\C,a),\textup{F}_{p-1,q},\zetaup}\wedge x$ and $\tth_{\tr(\C,d)}(y_{\tr(\C,a),\textup{F}_{p-1,q}}\wedge x)=-y_{\C,\textup{F}_{p,q}}$. So equation (\ref{dheq}) holds.
\end{enumerate}
\item[$\tau=\tg_{p,q}$, $p+q\ge 2$:]
This case is treated analogously to the case $\tau=\tf_{p,q}$.
\end{enumerate}

Let now $\textup{u}_\C=0$. We have to show that then the right hand side of equation~(\ref{dheq}) is trivial. To do this we have to consider two cases:
\begin{enumerate}
\item[-]
Let us have a situation as in Figure \ref{fsit12} for the first arc of $\C$. Then we have $\textup{act}(\tr(\C,a))=\textup{act}(\tr(\C,b))$.
\begin{figure}[htbp]
\centering
\includegraphics{gra.12}
\caption{}
\label{fsit12}
\end{figure}
We will check that
\begin{equation}\label{esit12}
\begin{aligned}
&(-1)^{1+n\textup{sp}(\C,b^0)}\eps(c)\td_{\tr(\C,b),\tau,\eps|_b}\circ\tth_{\tr(\C,c)}\\
=&(-1)^{\textup{sp}(\C,d^0)}\eps(d)\tth_{\tr(\C,d)}\circ\td_{\tr(\C,a),\tau,\eps|_a}
\end{aligned}
\end{equation}
holds if $\tr(\C,a)$ is of type $\tit$. In Situation 1 we have $\tsp(\C,b^0)=0$. In Situation 2 we have $\tsp(\C,b^0)=1$. Let $\theta$ be an allowed path of edges for $\tr(\C,a)$. In Situation 1 we have $\eps(c)\eps|_b(\theta)=(-1)^{n-1}\eps(d)\eps|_a(\theta)$ because $\tr(\C,s)$ is of type K for all $s\in\f(n,2)$ with $s_1=*$. In Situation 2 we have $\eps(c)\eps|_b(\theta)=(-1)^{n-1-\textup{sp}(\C,d^0)}\eps(d)\eps|_a(\theta)$ because for all $s\in\f(n,2)$ with $s_1=*$ and $e\in\f(n,1)$ with
\[e_i=\begin{cases}
0\qquad&\textup{for }i=1,\\
s_i\qquad&\textup{otherwise},
\end{cases}\]
we have: If $\tr(\C,e)$ is a split, then $\tr(\C,s)$ is of type A. If $\tr(\C,e)$ is a join, then $\tr(\C,s)$ is of type K. Obviously in Situation 2 equation (\ref{esit12}) holds. In Situation 1 equation~(\ref{esit12}) holds because of the grading rule.
\item[-]
Let us have none of the two situations in Figure \ref{fsit12} for the first arc of $\C$. We will show that $\tth_{\tr(\C,d)}\circ\td_{\tr(\C,a),\tau,\eps|_a}=0$ holds if $\tr(\C,a)$ is of type $\tau$.

If $\tr(\C,b)$ is of type $\sigma$ then for an edge assignment $\iota$ of type Y with respect to $\textup{m}(\C^*)$ it follows from $\tth_{\tr(\textup{m}(\C^*),d)}\circ\td_{\tr(\textup{m}(\C^*),a),\textup{m}(\sigma^*),\iota|_a}=0$ and the duality rule that we have $\td_{\tr(\C,b),\sigma,\eps|_b}\circ\tth_{\tr(\C,c)}=0$.

The first arc of $\C$ is called $\gammaup$. Note: The circles of a configuration are divided into segments by the boundary points of the arcs. We will treat each possibility for $\tau$ separately:
\begin{enumerate}
\item[$\tau=\ta_n$:]
Then $\gammaup$ connects two distinct segments of one of the two circles of $\tr(\C,a)$. But then $\tr(\C,d)$ is a join.
\item[$\tau=\tb_n$:]
Then $\gammaup$ connects two distinct circles of $\tr(\C,a)$. But then $\tr(\C,d)$ is a split.
\item[$\tau=\tc_{p,q}$:]
Then $\gammaup$ connects two distinct segments of the circle of $\tr(\C,a)$ where for at least one of these segments we have: The two arcs that bound this segment lie on the other side of the circle than $\gammaup$. But then $\tr(\C,d)$ is a join.
\item[$\tau=\tde_{p,q}$:]
Then $\tr(\C,d)$ is a split.
\item[$\tau=\tf_{p,q}$:]
Then there are two possibilities:
\begin{enumerate}
\item[1.]
For at least one of the two segments that $\gammaup$ connects we have: The segment is bounded by the starting point and the end point of the same arc. Then $\tr(\C,d)$ is a join, but $y_{\tr(\C,a),\textup{F}_{p,q}}$ is not divisible by at least one of the two circles of $\textup{act}(\tr(\C,d))$.
\item[2.]
Otherwise $\tr(\C,d)$ is a split, but $y_{\tr(\C,a),\textup{F}_{p,q}}$ is divisible by the active circle of $\tr(\C,d)$.
\end{enumerate}
\item[$\tau=\tg_{p,q}$:]
This case is treated analogously to the case $\tau=\tf_{p,q}$.
\end{enumerate}
\end{enumerate}

So we have proved Theorem \ref{dhth}.
\end{proof}
\begin{proof}[Proof of Theorem \ref{thind}]
It is enough to consider the case where $\C$ and $\D$ only differ in the orientation of the $i$-th arc. Let $f,g:\tc(\C)\to\tc(\C)$ be defined as $f\defeq\textup{id}_{\tc(\C)}+\tth_i(\C,\eps)$ and $g\defeq\textup{id}_{\tc(\C)}-\tth_i(\C,\eps)$. Then we have $f\circ g=\textup{id}_{\tc(\C)}=g\circ f$. Because of Theorem \ref{dhth} there is an edge assignment $\iota$ of type~Y with respect to $\D$ such that
\[\td(\C,\eps)-\td(\D,\iota)=\td(\C,\eps)\circ\tth_i(\C,\eps)-\tth_i(\C,\eps)\circ \td(\C,\eps).\]
It follows that
\begin{align*}
&\td(\D,\iota)\circ f(x)\\
=&\bigl(\td(\C,\eps)-\td(\C,\eps)\circ\tth_i(\C,\eps)+\tth_i(\C,\eps)
\circ\td(\C,\eps)\bigr)\bigl(x+\tth_i(\C,\eps)(x)\bigr)\\
=&\td(\C,\eps)(x)+\tth_i(\C,\eps)\circ\td(\C,\eps)(x)-\td(\C,\eps)
\circ\underbrace{\tth_i(\C,\eps)\circ\tth_i(\C,\eps)}_{=0}(x)\\
&+\underbrace{\tth_i(\C,\eps)\circ\td(\C,\eps)\circ\tth_i(\C,\eps)}_{=0}(x)\\
=&f\circ\td(\C,\eps)(x).
\end{align*}
So $\hc(\C,\eps)$ and $\hc(\D,\iota)$ are isomorphic. An isomorphism between $\hc(\D,\iota)$ and $\hc(\D,\etaup)$ is given by Ozsv\'ath, Rasmussen and Szab\'o \cite[Proof of Lemma 2.2]{ors}.
\end{proof}

\section{$d\circ d=0$}\label{sdd}
In this Section we sketch the proof of Theorem \ref{thdd} \cite{master}. It is enough to proof the following.
\begin{thm}\label{thdd2}
Let $\C$ be an $n$-dimensional oriented configuration and $\eps$ an edge assignment of type Y respective $\C$. Then we have
\[\sum_{i=1}^{n-1}\td_{n-i}(\C,\eps)\circ\td_i(\C,\eps)=0.\]
\end{thm}
To prove this we will use induction over $n$. For $n=2$ we get up to sign the odd Khovanov differential, so we have $\td_1(\C,\eps)\circ\td_1(\C,\eps)=-\partial(\C,\eps)\circ\partial(\C,\eps)=0$. The induction is completed by the following lemma and theorem:
\begin{lem}\label{lemdd}
Let $\C$, $\D$ be $n$-dimensional oriented configurations with $\overline{\C}=\overline{D}$. If Theorem \ref{thdd2} holds for all $(n-1)$-dimensional configurations, then for every edge assignment $\eps$ of type Y with respect to $\C$ there is an edge assignment $\etaup$ of type Y with respect to $\D$ such that
\[\sum_{i=1}^{n-1}\td_{n-i}(\C,\eps)\circ\td_i(\C,\eps)
=\sum_{i=1}^{n-1}\td_{n-i}(\D,\etaup)\circ\td_i(\D,\etaup).\]
\end{lem}
\begin{proof}
It is enough to consider the case where the orientation of the arcs only differs at the $j$-th arc for some $j\in\{1,\dots,n\}$. Because of Theorem \ref{dhth} there is an edge assignment $\etaup$ of type Y with respect to $\D$ such that
\[\td(\C,\eps)-\td(\D,\etaup)=\td(\C,\eps)\circ\tth_j(\C,\eps)-\tth_j(\C,\eps)
\circ\td(\C,\eps).\]
Then we have:
\begin{align*}
&\sum_{i=1}^{n-1}\td_{n-i}(\D,\etaup)\circ\td_i(\D,\etaup)\\
=&\sum_{i=1}^{n-1}\bigl(\td_{n-i}(\C,\eps)+\tth_j(\C,\eps)\circ\td_{n-i-1}(\C,\eps)
-\td_{n-i-1}(\C,\eps)\circ\tth_j(\C,\eps)\bigr)\\
&\circ\bigl(\td_i(\C,\eps)+\tth_j(\C,\eps)\circ\td_{i-1}(\C,\eps)-\td_{i-1}(\C,\eps)
\circ\tth_j(\C,\eps)\bigr)\\
=&\sum_{i=1}^{n-1}\td_{n-i}(\C,\eps)\circ\td_i(\C,\eps)
-\underbrace{\left(\sum_{i=1}^{n-1}\td_{n-i}(\C,\eps)
\circ\td_{i-1}(\C,\eps)\right)}_{=0}\circ\tth_j(\C,\eps)\\
&+\tth_j(\C,\eps)\circ\underbrace{\sum_{i=1}^{n-1}\td_{n-i-1}(\C,\eps)
\circ\td_{i}(\C,\eps)}_{=0}.\qedhere
\end{align*}
\end{proof}
\begin{thm}\label{thdd3}
Let $\D$ be an $n$-dimensional configuration and $\alpha\in\Lambda\V(\D)$, $\beta\in\Lambda\V(\D^*)$ monomials. Then there is an oriented configuration $\C$ such that $\overline{\C}=\D$ and such that for all edge assignments $\eps$ of type Y with respect to $\C$ the coefficient of
\[\sum_{i=1}^{n-1}\td_{n-i}(\C,\eps)\circ\td_i(\C,\eps)(\alpha)\]
at $\beta$ is trivial.
\end{thm}
Now Theorem \ref{thdd2} follows from Lemma \ref{lemdd} and Theorem \ref{thdd3} by induction over $n$.
\begin{proof}[Proof of Theorem \ref{thdd3}]
We only need to consider the case that $\D$ is active and $n\ge 3$.

First let $\D$ be disconnected. Then it is enough to consider the case where $\D$ consists of exactly two connected components and for $k\in\{1,\dots,n-1\}$ the first $k$ arcs of $\D$ lie in one connected component and the other $n-k$ arcs lie in the other connected component. Let $a,b\in\f(n,n-k)$ with
\[a_j=\begin{cases}
0\qquad&\textup{for }j\le k,\\
*\qquad&\textup{for }j>k,
\end{cases}\]
and
\[b_j=\begin{cases}
1\qquad&\textup{for }j\le k,\\
*\qquad&\textup{for }j>k.
\end{cases}\]
Let $c,d\in\f(n,k)$ with
\[c_j=\begin{cases}
*\qquad&\textup{for }j\le k,\\
0\qquad&\textup{for }j>k,
\end{cases}\]
and
\[d_j=\begin{cases}
*\qquad&\textup{for }j\le k,\\
1\qquad&\textup{for }j>k.
\end{cases}\]
We have to show that for all $\C$, $\eps$ we have:
\begin{equation}\label{ddeq}
(-1)^{k+(n-k+1)\textup{sp}(\C,b^0)}\td_{\tr(\C,b),\eps|_b}\circ\td_{\tr(\C,c),\eps|_c}
+(-1)^{k-n+(k+1)\textup{sp}(\C,d^0)}\td_{\tr(\C,d),\eps|_d}
\circ\td_{\tr(\C,a),\eps|_a}=0
\end{equation}
We have $\textup{act}(\tr(\C,a))=\textup{act}(\tr(\C,b))$ and $\textup{act}(\tr(\C,c))=\textup{act}(\tr(\C,d))$. Let $\theta$ and $\zetaup$ be allowed paths of edges for $\tr(\C,a)$ and $\tr(\C,c)$ respectively. Then we have
\[\eps|_b(\theta)\eps|_c(\zetaup)
=(-1)^{k(n-k)-\textup{sp}(\C,b^0)\textup{sp}(\C,d^0)}
\eps|_d(\zetaup)\eps|_a(\theta).\]
Let $\tr(\C,a)$, $\tr(\C,c)$ be of type $\tau,\sigma\in\textup{T}$. Then we have
\begin{align*}
&\td_{\tr(\C,b),\tau,\eps|_b}\circ\td_{\tr(\C,c),\sigma,\eps|_c}(x_{\tr(\C,c),
\sigma,\zetaup}\wedge x_{\tr(\C,b),\tau,\theta})\\
=&\eps|_c(\zetaup)\textup{sp}(\tr(\C,c),\zetaup)\td_{\tr(\C,b),\tau,\eps|_b}
(y_{\tr(\C,c),\sigma}\wedge x_{\tr(\C,b),\tau,\theta})\\
=&(-1)^{\textup{gr}(y_{\tr(\C,c),\sigma})\textup{gr}(x_{\tr(\C,b),\tau,\theta})}
\eps|_b(\theta)\eps|_c(\zetaup)\textup{sp}(\tr(\C,b),\theta)\textup{sp}(\tr(\C,c),
\zetaup)
y_{\tr(\C,b),\tau}\wedge y_{\tr(\C,c),\sigma}
\end{align*}
and
\begin{align*}
&\td_{\tr(\C,d),\sigma,\eps|_d}\circ\td_{\tr(\C,a),\tau,\eps|_a}(x_{\tr(\C,c),
\sigma,\zetaup}\wedge x_{\tr(\C,b),\tau,\theta})\\
=&(-1)^{\textup{gr}(x_{\tr(\C,c),\sigma,\zetaup})\textup{gr}(x_{\tr(\C,b),
\tau,\theta})}
\eps|_a(\theta)\textup{sp}(\tr(\C,a),\theta)\td_{\tr(\C,d),\sigma,\eps|_d}
(y_{\tr(\C,a),\tau}\wedge x_{\tr(\C,d),\sigma,\zetaup})\\
=&(-1)^{\textup{gr}(x_{\tr(\C,c),\sigma,\zetaup})(\textup{gr}(x_{\tr(\C,b),
\tau,\theta})
+\textup{gr}(y_{\tr(\C,b),\tau}))}
\eps|_d(\zetaup)\eps|_a(\theta)\textup{sp}(\tr(\C,d),\zetaup)
\textup{sp}(\tr(\C,a),\theta)
y_{\tr(\C,d),\sigma}\wedge y_{\tr(\C,a),\tau}\\
=&(-1)^{(\textup{gr}(x_{\tr(\C,c),\sigma,\zetaup})+\textup{gr}(y_{\tr(\C,c),\sigma}))
(\textup{gr}(x_{\tr(\C,b),\tau,\theta})+\textup{gr}(y_{\tr(\C,b),\tau}))
+k(n-k)-\textup{sp}(\C,b^0)\textup{sp}(\C,d^0)}\\
&\cdot\td_{\tr(\C,b),\tau,\eps|_b}
\circ\td_{\tr(\C,c),\sigma,\eps|_c}(x_{\tr(\C,c),\sigma,\zetaup}\wedge x_{\tr(\C,b),\tau,\theta}).
\end{align*}
Because of the grading rule we have:
\begin{align*}
&(\textup{gr}(x_{\tr(\C,c),\sigma,\zetaup})+\textup{gr}(y_{\tr(\C,c),\sigma}))
(\textup{gr}(x_{\tr(\C,b),\tau,\theta})+\textup{gr}(y_{\tr(\C,b),\tau}))\\
\equiv&(\textup{sp}(\C,b^0)-k+1)(\textup{sp}(\C,d^0)-(n-k)+1)\qquad(\textup{mod }2)
\end{align*}
Now equation (\ref{ddeq}) follows.

The proof in the case that $\D$ is connected uses similar methods \cite[Satz~76]{master}; it is a refined version of the corresponding proof in the mod 2 case by Szab\'o \cite[Theorem 6.3]{szabo}.
\end{proof}

\bibliographystyle{amsplain}

\bigskip

\end{document}